\pdfoutput=1
\documentclass{article}
\usepackage{amsthm}
\usepackage{amsmath}
\usepackage{amssymb}
\usepackage{amscd}
\usepackage{graphicx}

\newtheorem{theorem}[subsection]{Theorem}
\newtheorem{lemma}[subsection]{Lemma}

\newtheorem{corollary}[subsection]{Corollary}

\theoremstyle{definition}

\newtheorem{definition}[subsection]{Definition}

\theoremstyle{remark}

\newtheorem{remark}[subsection]{Remark}

\def\R{\mathbb{R}}

\vspace{1.5cm}
\title{{\bf Newton flows for elliptic functions II}\\
{\bf {\small Structural stability:
Classification \& Representation}}}

\author{G.F. Helminck,\\
Korteweg-de Vries Institute\\
University of Amsterdam\\
P.O. Box 94248\\
1090 GE Amsterdam\\
The Netherlands\\
e-mail: g.f.helminck@uva.nl\\ 
F. Twilt,\\
Department of Applied Mathematics\\
University of Twente\\
P.O. Box 217, 7500 AE Enschede\\
The Netherlands\\
e-mail: f.twilt@kpnmail.nl\\
}
 
\begin{document}
\maketitle

\begin{abstract}
\noindent
In our previous paper we associated to each non-constant elliptic function $f$ on a torus $T$ a dynamical system, the elliptic Newton flow corresponding to $f$. We characterized the functions for which these flows are structurally stable and 
showed a genericity result. In the present paper we focus on the classification and representation of these structurally stable flows.

The phase portrait of a structurally stable elliptic Newton flow generates a connected, cellularly embedded, graph $\mathcal{G}(f)$ on a torus $T$ with $r$ vertices, 2$r$ edges and $r$ faces that fulfil certain combinatorial properties ({\it Euler, Hall}) on some of its subgraphs. The graph $\mathcal{G}(f)$ determines the conjugacy class of the flow. [{\it classification}]

A connected, cellularly embedded toroidal graph $\mathcal{G}$ with the above {\it Euler} and {\it Hall} properties, is called a {\it Newton graph}. Any Newton graph $\mathcal{G}$ can be realized as the graph $\mathcal{G}(f)$ of the structurally stable Newton flow for some function $f$. 
  
This leads to: up till conjugacy between flows and (topological) equivalency between graphs, there is a 
one to one correspondence between the structurally stable Newton flows and Newton graphs, both with respect to the same order $r$ of the underlying functions $f$.[{\it representation}]
      
Finally, we clarify the analogy between rational and elliptic Newton flows, and show
that the detection of elliptic Newton flows is possible in polynomial time.
      
The proofs of the above results rely on Peixoto's characterization/classification theorems for structurally stable dynamical systems on compact 2-dimensional manifolds, Stiemke's theorem of the alternatives, Hall's theorem of distinct representatives, the Heffter-Edmonds-Ringer rotation principle for embedded graphs, an existence theorem on gradient dynamical systems by Smale, and an interpretation of Newton flows as steady streams. 
\end{abstract}

\noindent
{\bf Subject classification:} 
05C45, 05C75, 30C15, 30D30, 30F99,  33E05, 34D30, 37C15, 37C20, 37C70, 49M15, 68Q25.\\

\noindent
{\bf Keywords:} Dynamical system (gradient-), desingularized Newton flow (rational, elliptic), structural stability, elliptic function (Jacobian, Weierstrass), phase portrait, Newton graph (elliptic-, 
rational-), cellularly embedded toroidal (distinguished) graph, face traversal procedure, steady stream, complexity, Angle property, Euler property, Hall condition.

\section{Elliptic Newton flows: a recapitulation}
\label{ENFll.1}

In order to clarify the context of the present paper, we recapitulate some earlier results.\\

\noindent
{\large{\bf {\small 1.1 Elliptic Newton flows on the plane and on a torus}}}\\

Let $f$  be an elliptic (i.e., meromorphic, doubly periodic) function of order $r(\geqslant 2)$ on the complex plane $\mathbb{C}$ with $(\omega_{1},\; \omega_{2})$, ${\rm Im}\frac{\omega_{2}}{\omega_{1}} >0,$ as basic periods spanning a lattice $\Lambda(=\Lambda_{\omega_{1},\; \omega_{2}})$.

The {\it planar elliptic Newton flow} $\overline{\mathcal{N} }(f)$ is a 
 $C^{1}$-vector field on $\mathbb{C}$, defined as a {\it desingularized version}\footnote{\label{FTN1}
 In fact, we consider the system $  \dfrac{dz}{dt} =-(1+|f(z)|^{4})^{-1}|f'(z)|^{2}\dfrac{f (z)}{f^{'} (z)}$: a continuous version of Newton's damped iteration method for finding zeros for $f$.
} of the planar dynamical system, $\mathcal{N}(f)$, given by: (cf. \cite{HT1})
\begin{equation}
\label{vgl2x}
  \dfrac{dz}{dt} = \dfrac{-f (z)}{f^{'} (z)}, z \in \mathbb{C}.
\end{equation}

On a  non-singular, oriented $\overline{\mathcal{N} }(f)$-trajectory $z(t)$ we have: (cf. \cite{HT1})\\
\noindent
- arg $(f)=$constant and $|f(z(t))|$ is a strictly decreasing function on $t$.\\
\noindent
So that an $\overline{\mathcal{N} }(f)$-{\it equilibrium} is:
\ \\
\noindent
-  Attractor, or repellor, or saddle; see the Comment on Fig.1, where $N(f)$, $P(f)$ and
    $C(f))$ stand for resp. the set of zeros, poles and critical points for $f$.\\
    
\begin{figure} [htbp]
\begin{center}
\includegraphics*[height=4cm, width=14cm]{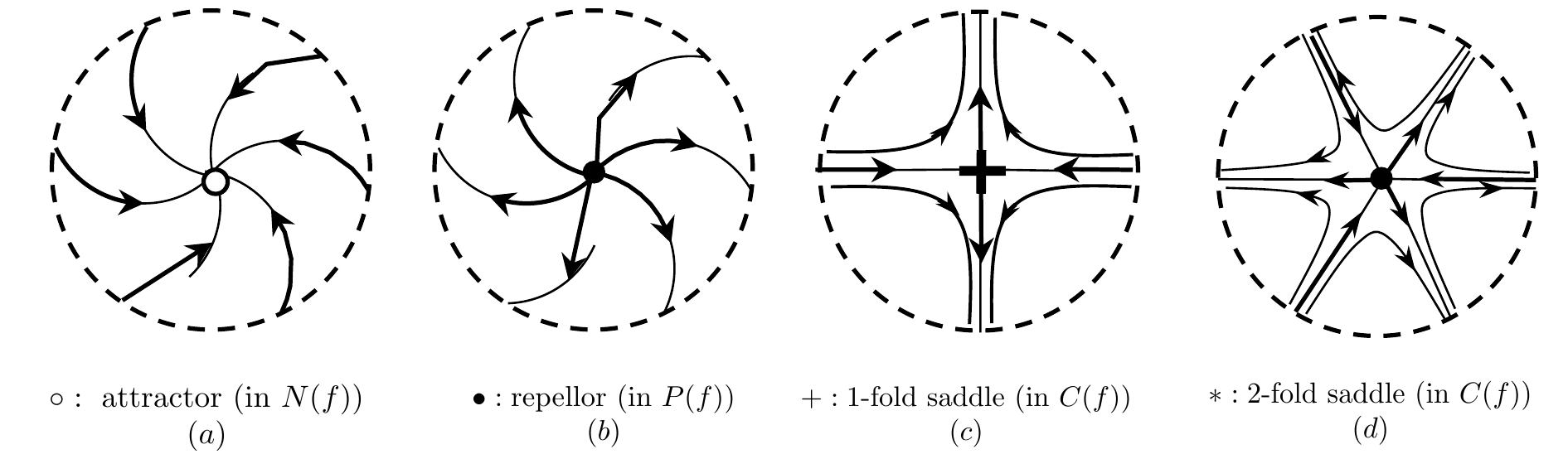}
\caption {Local phase portraits around equilibria of $\overline{\mathcal{N} }(f)$ }
\label{Figure1}
\end{center}
\end{figure}

\noindent
\underline{Comment on Fig. \ref{Figure1}}:\\
\noindent
Fig. \ref{Figure1}-$(a),(b)$:
In a $k$-fold zero (pole) for $f$ the flow $\overline{\mathcal{N} }(f)$ exhibits a {\it stable (unstable) star node}
and  each (principal) value of arg$f$ appears precisely $k$ times on equally distributed incoming (outgoing) trajectories. Moreover,  two different incoming (outgoing) trajectories intersect under a {\it non vanishing} angle $\!\frac{\Delta}{k}$, where $\Delta$ stands for the difference of the arg$f$-values on these trajectories. \\
\noindent
Fig.\ref{Figure1}-$(c),(d)$: 
In case of a $k$-fold critical point (i.e. a $k$-fold zero for $f'$, no zero for $f$) the flow $\overline{\mathcal{N} }(f)$ exhibits a $k$-fold saddle, the stable (unstable) separatrices being equally distributed around this point. The two unstable (stable) separatrices at a 1-fold saddle, see Fig.\ref{Figure1}(c), constitute the ``local'' {\it unstable (stable) manifold} at this saddle point.\\

Functions of the type $f$ correspond to the meromorphic functions on the complex torus $T(\Lambda)$($=\mathbb{C} / \Lambda_{\omega_{1},\; \omega_{2}}$).
So, we can interprete $\overline{\mathcal{N} }(f)$ as a global  $C^{1}$-vector field, denoted\footnote{\label{FTN2}Occasionally, we will refer to $\overline{\overline{\mathcal{N}} }(f)$ as to a {\it toroidal Newton flow}.} $\overline{\overline{\mathcal{N}} }(f)$, on the Riemann surface $T(\Lambda)$  and it is allowed to apply results for  $C^{1}$-vector fields on compact differential manifolds, such as certain theorems of Poincar\'e-Bendixon-Schwartz on limiting sets and those of Baggis-Peixoto on $C^{1}$-structural stability. In particular, the local phase portraits around $\overline{\overline{\mathcal{N}} }(f)$-equilibriae are as in Fig.\ref{Figure1}. \\

\noindent
{\large{\bf {\small 1.2 The canonical form for a toroidal Newton flow; the topology $\tau_{0}$}}}\\

It is well-known that the function $f$ has precisely $r$ zeros and $r $ poles (counted by multiplicity) on the half open\! / \!half closed period parallelogram $P(=P_{\omega_{1},\omega_{2}})$ given by $\{t_{1}\omega_{1} + t_{2}\omega_{2} \mid 0 \leqslant t_{1}<1, \;0 \leqslant t_{2}<1\}$.

Denoting these zeros and poles by
$a_{1}, \! \cdots \!\!,a_{r} $, resp. $b_{1}, \! \cdots \! ,b_{r} $, 
we have: (cf. \cite{HT1}),
\cite{M2})
\begin{equation}
\label{Vgl2}
a_{i} \neq b_{j},i, j=1, \! \cdots \! \!,r \text{ and }a_{1}+  \cdots +a_{r} =b_{1} + \cdots  +b_{r} \text{ mod }\Lambda.
\end{equation}
and thus
\begin{equation}
\label{Vgl3}
[a_{i}] \neq [b_{j}],i, j=1, \! \cdots \! ,r \text{ and }[a_{1}]+ \! \cdots \! +[a_{r}] =[b_{1}] +\! \cdots \! +[b_{r}],
\end{equation}
where $[a_{1}], \! \cdots \!, [a_{r}] $ and  $[b_{1}], \! \cdots \!, [b_{r}] $ are the zeros resp. poles for $f$ on $T(\Lambda)$ and $[\cdot]$ stands for the congruency class $\mathrm {mod} \, \Lambda$
of a number in $\mathbb{C}$.

\begin{theorem}$($The canonical form for toroidal Newton flows$)$\\
\label{T1.1}
\begin{itemize}
\item Given a flow $\overline{\overline{\mathcal{N}} }(f)$ on $T(\Lambda)$, there exists an elliptic function $f^{*}$ of order $r$ with period lattice $\Lambda^{*}(=\Lambda_{1, i})$
together with a homeomorphism $T(\Lambda) \to T(\Lambda^{*})$
mapping the phase portraits of $\overline{\overline{\mathcal{N}} }(f)$ and $\overline{\overline{\mathcal{N}} }(f^{*})$ onto each other, thereby respecting the orientations of the trajectories.
\item Moreover:
If $a^{*}_{1}, \! \cdots \!\!,a^{*}_{r} $, resp. $b^{*}_{1}, \! \cdots \! ,b^{*}_{r} $ are the zeros and poles of $f^{*}$ in $P^{*}(=P_{1,i})$, then 
\begin{equation*}
f^{*}(z)=\frac{\sigma(z - a^{*}_1)\! \cdots \! \sigma (z - a^{*}_r)} {\sigma(z - b^{*}_1) \! \cdots \! \sigma(z - b^{*}_{r-1})  \sigma (z - (b^{*}_{r})^{'})}, 
(b^{*}_{r})^{'} = a^{*}_{1}+ \! \cdots \! +a^{*}_{r} - b^{*}_{1}- \! \cdots \! -b^{*}_{r-1},
\end{equation*}
where $\sigma$ stands for the Weierstrass' sigma function w.r.t. the lattice $\Lambda^{*}$. \\
\item
Conversely:
 If $c_{1}, \! \cdots \!\!,c_{r} $, resp. $d_{1}, \! \cdots \! ,d_{r} $
 stands for any pair of $r$ tuples in $P^{*}$ that fulfil the relations (\ref{Vgl2}), then due to the
  basic properties of the quasi periodic function $\sigma$,  a function of the form
  \begin{equation*}
 \frac{\sigma(z - c_1)\! \cdots \! \sigma (z - c_r)} {\sigma(z - d_1) \! \cdots \! \sigma(z - d_{r-1})  \sigma (z - d_{r}^{'})}, \text{ with }d_{r}^{'} = c_{1}+ \! \cdots \! +c_{r} - d_{1}- \! \cdots \! -d_{r-1},
\end{equation*}
is elliptic w.r.t. to $\Lambda^{*}$  
with $[c_{1}], \cdots ,[c_{r}]$ resp. $[d_{1}], \cdots ,[d_{r}]$ as zeros, poles on $T(\Lambda^{*})$.
\end{itemize}
\end{theorem}

 Now, it is not difficult to see that the elliptic functions of order $r$, and also the underlying toroidal Newton flows, can be represented by the set of all ordered pairs 
 $$
 (\{ [c_{1}],\! \cdots \! ,[c_{r} ]\}, \, \{ [d_{1}], \! \cdots \! ,[d_{r} ]\})
$$ 
of congruency classes $\mathrm {mod} \, \Lambda^{*}$ with $c_{i}, d_{i} \in P^{*}, i=1, \dots, r,$ that fulfil 
(\ref{Vgl3}). This representation space can be endowed with a topology, say $\tau_{0}$, that
is induced by the Euclidean topology on $\mathbb{C}$, and is natural in the following sense: (cf. \cite{HT1},\cite{HTtotaal})\\
\noindent
Given an elliptic function $f $ of order $r$
and $\varepsilon >0$ sufficiently small, a $\tau_{0}$-neighbourhood $\mathcal{O}$ of $f$ exists such that for any $g \in \mathcal{O}$ , the zeros (poles) for $g$ are contained in $\varepsilon$-neighbourhoods of the zeros (poles) for $f$.\\

\noindent
{\large{\bf {\small 1.3 Structural stability}}}\\

Let $E_{r}(\Lambda)$ be the set of all elliptic functions $f$ of order $r$ on the torus $T(\Lambda)=\mathbb{C} / \Lambda$ and $N_{r}(\Lambda)$ the set of all toroidal Newton flows $\overline{\overline{\mathcal{N}} }(f)$. We assume (no loss of generality; see the above Subsection 1.2) that  $\Lambda=\Lambda_{1, i}$
,  and write :  $E_{r}(\Lambda)=E_{r}$,  $T(\Lambda)=T$ and $N_{r}(\Lambda)=N_{r}$. \\

By $X(T)$ we mean the set of all $C^{1}$-vector fields on $T$, endowed with the $C^{1}$-topology. 
The topology $\tau_{0}$ on $E_{r}$
and the $C^{1}$-topology  on $X(T)$ are matched by: 
(cf. \cite{HT1})
\begin{lemma}
\label{L1.1}
The map $E_{r} \rightarrow X(T): f \mapsto \overline{\overline{\mathcal{N}}} (f)$
is $\tau_{0}\!-\!C^{1}$ continuous.
\end{lemma}

Two flows $\overline{\overline{\mathcal{N}} }(f)$ and $\overline{\overline{\mathcal{N}} }(f)$) in $N_{r}$ are called {\it conjugate}, denoted $\overline{\overline{\mathcal{N}}} (f) \sim \overline{\overline{\mathcal{N}}} (g)$, if there is a homeomorphism from $T$ onto itself mapping maximal trajectories of $\overline{\overline{\mathcal{N}} }(f)$ onto those of $\overline{\overline{\mathcal{N}} }(g)$, thereby respecting the orientations of these trajectories.

We call the flow $\overline{\overline{\mathcal{N}}} (f) $ 
{\it $\tau_{0}$-structurally stable}, if there is a $\tau_{0}$-neighborhood $\mathcal{O}$ of $f$, such that for all $g \in \mathcal{O}
$ we have: $\overline{\overline{\mathcal{N}}} (f) \sim \overline{\overline{\mathcal{N}}} (g)$. 
The set of all structurally stable Newton flows $\overline{\overline{\mathcal{N}}} (f)$
 is denoted by $\tilde{N}_{r}$.

By Lemma \ref{L1.1} it  follows: $C^{1}$-structural stability for $\overline{\overline{\mathcal{N}}} (f)$ implies $\tau_{0}$-structural stability for $\overline{\overline{\mathcal{N}}} (f)$. So, when discussing structural stable toroidal Newton flows we will skip the adjectives $\tau_{0}$  and $C^{1}$\\.  

We proved: (cf. \cite{HT1})

\begin{theorem} $($Characterization and Genericity of structural stability$)$
\label{T1.2} 
\begin{enumerate}
\item[$($1$)$]  
$\overline{\overline{\mathcal{N}}} (f) \in \tilde{N}_{r}$
 if and only if the function $f$  is {\it non-degenerate}, i.e., all zeros, poles and critical
      points for $f$ are simple, and no critical points for $f$ are connected by $\overline{\overline{\mathcal{N}}} (f)$-trajectories.
\item[$($2$)$] 
The set of all non-degenerate functions of order $r$ is open and dense in $E_{r}$.
\end{enumerate}
\end{theorem}

\begin{figure}[h!]
\begin{center}
\includegraphics[scale=0.45]{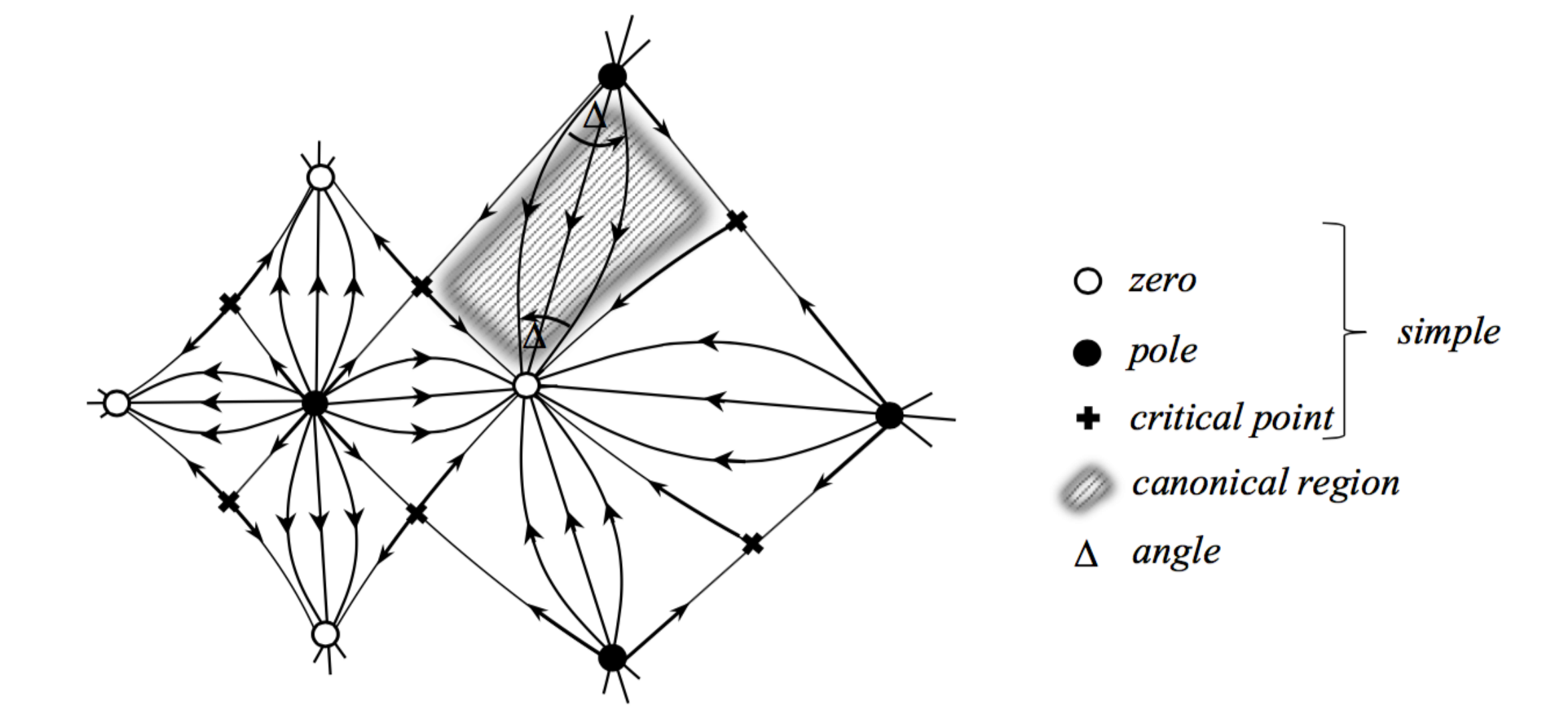}
\caption{\label{Figure2N} Basin of repulsion(attraction) in the phase portrait of $\overline{\overline{\mathcal{N}}} (f)$ for a pole(zero) of $f$.}
\end{center}
\end{figure}

\newpage
We list some properties that will play a role in the sequel, see 
Comment on Fig.\ref{Figure1}:
\begin{lemma}$($Properties of structurally stable toroidal Newton flows $\overline{\overline{\mathcal{N}}} (f)$$)$
\label{T1.3}
\begin{enumerate}
\item[$(a)$] If $\overline{\overline{\mathcal{N}}} (f)$ is structurally stable, then also $\overline{\overline{\mathcal{N}}} (\frac{1}{f})$, and  $\overline{\overline{\mathcal{N}}} (\frac{1}{f})=-\overline{\overline{\mathcal{N}}} (f)$. $[${\bf Duality}$]$
\item[$(b)$] 
There are precisely $2r$ orthogonal saddles for $\overline{\overline{\mathcal{N}}} (f)$.
\item[$(c)$] The boundary of the basin of a 
repellor(attractor) is made up by the unstable (stable) manifolds at the saddles situated in this boundary. $($cf. Fig.\ref{Figure2N}$)$
\end{enumerate}
\end{lemma}

\noindent
As an illustration we present in Fig.\ref{Figure3N} and \ref{Figure4N}
planar/toroidal Newton flows for Jacobian functions ${\rm sn}_{\omega_{1}, \omega_{2}}$ with only simple attractors, repellors and saddles; see also \cite{A/S}, \cite{HT1}
and the forthcoming Remark \ref{R6.10}.
For more examples of (structurally stable) Newton flows, see \cite{HT3}.

\begin{figure}[h!]
\centering
\includegraphics[width=5.7in]{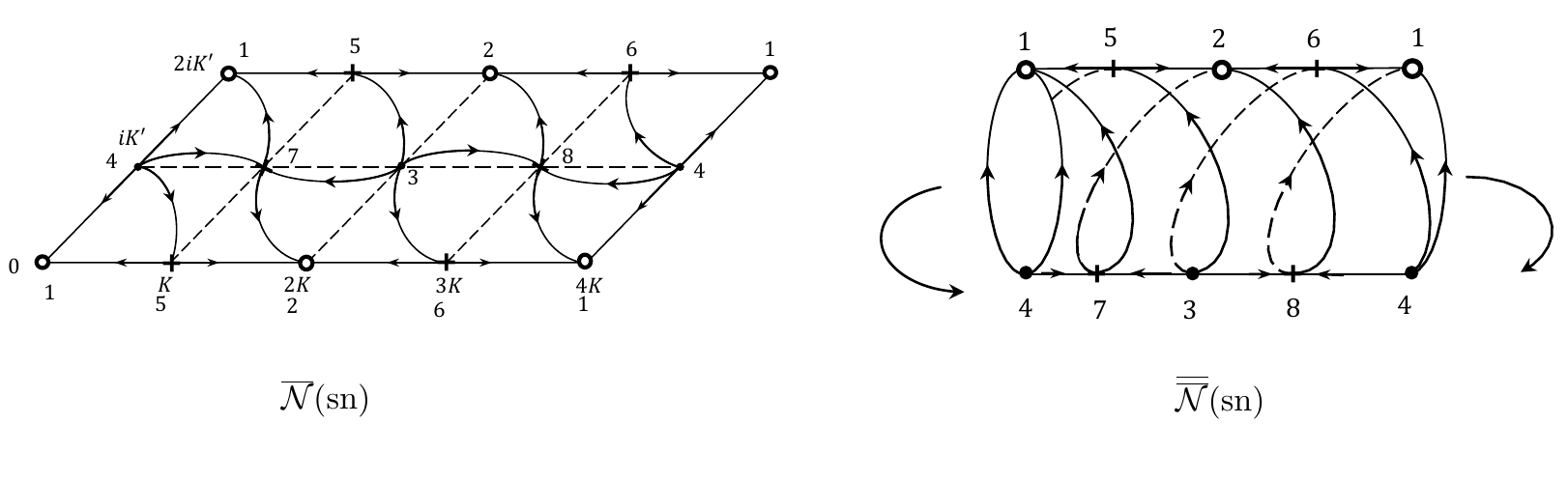}
\caption{ Planar and toroidal Newton flows for ${\rm sn}_{\omega_{1},\omega_{2}}$; structurally stable.}
\label{Figure3N}
\end{figure}
\begin{figure}[h!]
\centering
\includegraphics[width=2in]{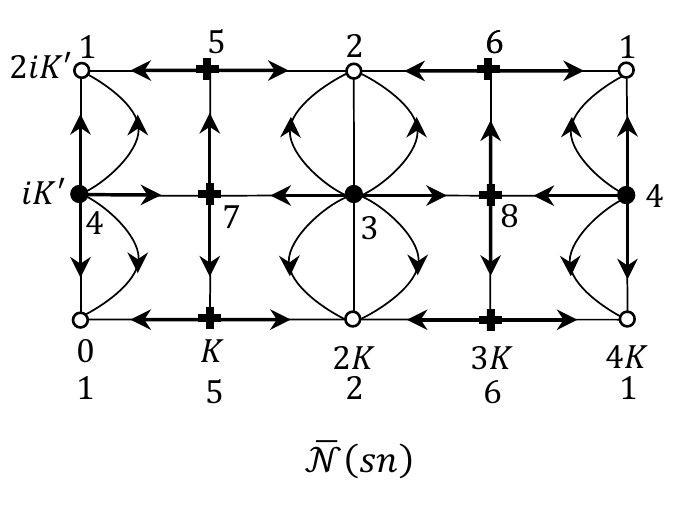}
\caption{Planar Newton flow for ${\rm sn}_{\omega_{1},\omega_{2}}$; not structurally stable.}
\label{Figure4N}
\end{figure}

\newpage
\noindent
{\large{\bf {\small 1.4 Toroidal vs. rational Newton flows; purpose of the paper}}}\\

If we choose for $f$ {\it rational} functions (meromorphic on the Riemann sphere $S^{2}$), we obtain the class of so-called {\it spherical Newton flows}. These flows have many concepts/features in common with (the class of) toroidal Newton flows and are already studied before (cf. \cite{JJT1},\cite{JJT2},\cite{JJT3},\cite{JJT4})
; in fact, {\it characterization} \& {\it genericity results}, 
analogous with Theorem \ref{T1.2}, have been proved. Moreover, spherical Newton flows can be {\it classified} \& {\it represented} in terms of certain sphere graphs (i.e. the ``principal parts'' of the phase portraits of structurally stable spherical Newton flows). The target of the present paper is to prove such a
{\it classification} \& {\it representation result} for toroidal Newton flows.

\section{Structurally\,stable\,elliptic\,Newton\,flows:\,Classification}
\label{sec6}

In this section, let $f$ 
be non-degenerate 
of order $r$, thus $\overline{\overline{\mathcal{N}}} (f)$ is structurally stable.

\noindent
Now, the following definition makes sense: (cf. Subsection 1.1 and Lemma 
\ref{T1.3})

\begin{definition}
\label{D6.1}
The 
graph $\mathcal{G}(f )$, $f \in \tilde{E}_{r}$, on the torus $T$ is given by:
\begin{itemize}
\item 
Vertices
are the $r$ zeros for $f $ on $T$
(as attractors for $\overline{\overline{\mathcal{N}} }(f)$).
\item 
Edges 
are the 2$r$ unstable manifolds at the critical points for $f$ on $T$ as  
$\overline{\overline{\mathcal{N}} }(f)$-saddles. 
\end{itemize}
\end{definition}

\noindent
Note that the faces of $\mathcal{G}(f )$ are precisely the $r$ {\it basins of repulsion} of the poles, say
$[b_{j}]$, $j=1, \! \cdots \!, r$ for $f$
 on $T$ (as repellors for $\overline{\overline{\mathcal{N}} }(f)$) and
will be denoted by $F_{b_{j}}(f)$; 
their boundaries by $\partial F_{b_{j}}(f)$.
These boundaries, consisting of {\it unstable manifolds} at saddles
for $\overline{\overline{\mathcal{N}} }(f)$,  are 
subgraphs of $\mathcal{G}(f )$, see Fig.\ref{Figure2N}\\

\noindent 
Analogously, we define the graph\footnote{\label{vtn1s6}$ \mathcal{G}(f )$ and $\mathcal{G}^{*}(f )$ are {\it geometrical duals}; see also Section \ref{Nsec7}.}, say $\mathcal{G}^{*}(f )$, on the poles and the stable $\overline{\overline{\mathcal{N}} }(f)$-manifolds at the critical points for $f$
 on $T$.
\begin{lemma}
\label{L6.2N}
Both $\mathcal{G}(f )$ and $\mathcal{G}^{*}(f )$
are multigraphs\footnote{ i.e., multiple edges are allowed, but no loops (cf. \cite{Har}); 
note however that the concept of multigraph in (\cite{MoTh}) includes loops.}
embedded in $T$.
\end{lemma}
\begin{proof}
If $\mathcal{G}(f )$ would have a loop, the two unstable $\overline{\overline{\mathcal{N}} }(f)$-separatrices at some critical point for $f$ would approach the same zero, say $[a]$, on $T$. In that case, the zeros (simple!) for $f$ in the plane, corresponding to $[a]$, will then be approached by two {\it different} trajectories (of the planar version $\overline{\mathcal{N}}(f)$) 
with the same value of arg $f$. This is impossible (cf.
the Comment on Fig.\ref{Figure1}).
The second part of the assertion follows by interchanging the roles of the poles and zeros for $f$.
\end{proof}
\begin{corollary}
\label{C6.3}
An edge in $\mathcal{G}(f )$ or $\mathcal{G}^{*}(f )$ is contained in the boundaries of two different 
faces.
\end{corollary}

\noindent
Next we introduce a graph on $T$, denoted $\mathcal{G}(f) \wedge  \mathcal{G}^{*}(f )$,  which may be considered as the ``common refinement of $\mathcal{G}(f))$ and $\mathcal{G}^{*}(f )$'':
 
\begin{definition}
\label{D6.5}
The {\it vertices} of $\mathcal{G}(f) \wedge \mathcal{G}^{*}(f )$ are defined as the zeros, poles and critical points for $f$, whereas the {\it edges} are the stable and unstable separatrices of $\overline{\overline{\mathcal{N}} }(f)$ at the critical points for $f $. 
\end{definition}

The faces of $\mathcal{G}(f) \wedge \mathcal{G}^{*}(f )$ are the so-called {\it canonical regions} for $\overline{\overline{\mathcal{N}} }(f)$, i.e. the connected components of what is left after deleting from $T$ all the $\overline{\overline{\mathcal{N}} }(f)$-equilibria and all stable and unstable manifolds at the saddles of $\overline{\overline{\mathcal{N}} }(f)$. A priori, the canonical regions of a $C^{1}$-structurally stable flow on $T$ (without closed orbits) are of one of the Types 1,2,3 in Fig. \ref{Figure12} (cf. Fig.\ref{Figure2N} and \cite{Peix2}). However, by Lemma \ref{L6.2N} the flow 
$\overline{\overline{\mathcal{N}} }(f)$- although structurally stable - cannot admit canonical regions of Types 2 and 3.

\begin{figure}[h]
\begin{center}
\includegraphics[scale=0.9]{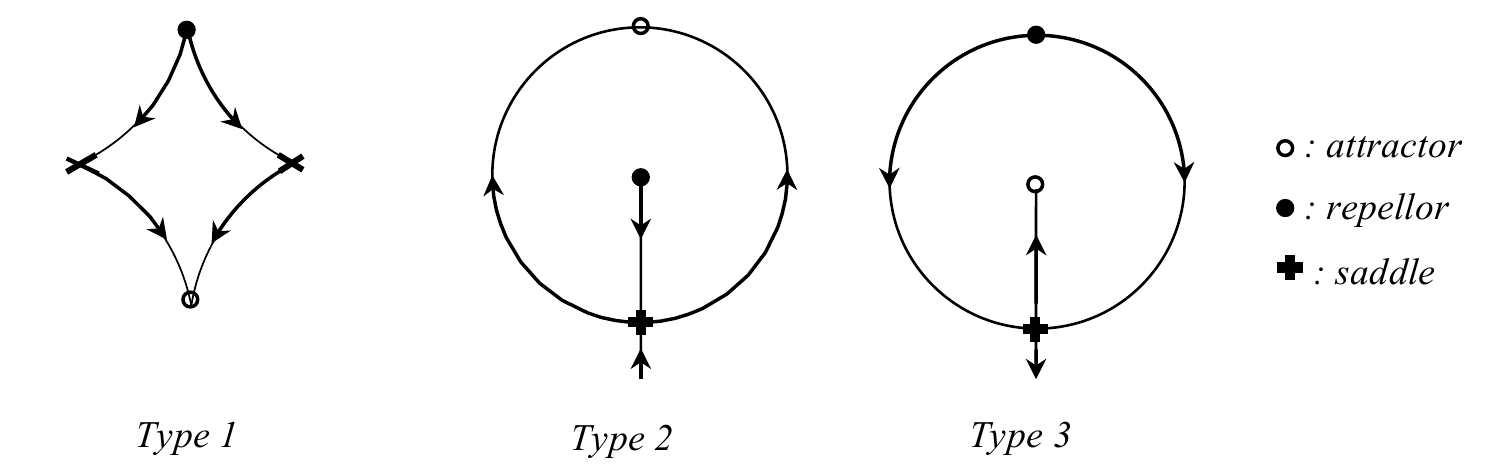}
\caption{\label{Figure12}The canonical regions of a structural stable flow on $T$}
\end{center}
\end{figure}

So. we only have to deal with canonical regions of Type 1. Since all zeros, poles and critical points for $f$ are simple, we find: (see 
Subsection 1.1)

\begin{lemma}
\label{L6.5N}
In a canonical region of $\overline{\overline{\mathcal{N}} }(f)$, the angles (anti-clockwise measured) at the pole and the zero are well-defined, strictly positive and equal.
\end{lemma}		      

\noindent
Since a face $ F_{b_{j}}\!(f)$ 
is built up from all canonical 
regions 
that have $[b_{j}]$ 
in common, we find: 
\begin{corollary}
\label{C6.6N}
All (anti-clockwise measured)
angles spanning a sector of $ F_{b_{j}}\!(f)$
at the vertices in its boundary, 
 are non-vanishing and sum up to $2 \pi$. 	
\end{corollary}

\begin{figure}[h]
\begin{center}
\includegraphics[scale=0.8]{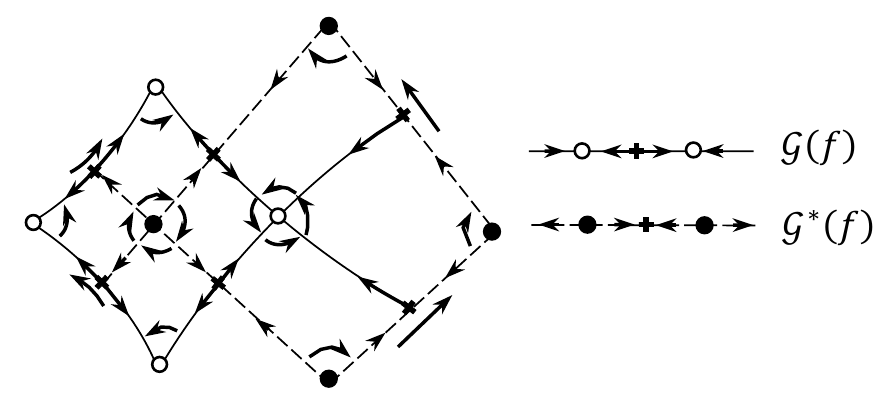}
\caption{\label{Figure13} Oriented facial walks on $\mathcal{G}(f )$ and $\mathcal{G}^{*}(f )$.}
\end{center}
\end{figure}

\newpage
\begin{lemma}
\label{L6.7N}
Each subgraph $\partial F_{b_{j}}\!(f)$ is 
Eulerian\footnote{\label{FN13}i.e. the 
graph $\partial F_{b_{j}}(f)$ 
admits a so-called {\it Euler trail}: a closed walk that traverses each
  edge exactly once and goes through all vertices. We do not distinguish between an Euler trail and its cyclic shift}.
\end{lemma}
\begin{proof}
Traverse 
the set of all canonical regions centered at $[b_{j}]$ once. In this way we determine a closed walk, say $w_{b_{j}}$, through all the vertices and edges of $\partial F_{b_{j}}(f)$
; see Fig. \ref{Figure13}. By Corollary \ref{C6.3}, this walk contains each edge of $\partial F_{b_{j}}(f)$ only once (since otherwise the two stable separatrices at the saddle on such an edge must originate from $[b_{j}]$). So, $w_{b_{j}}$  is the desired Euler trail.			                    
\end{proof}

\noindent
The  walk $w_{b_{j}}$  in the above proof will be referred to as to the {\it facial walk} for $\partial F_{b_{j}}(f)$. Analogously, we define the (Eulerian!) facial walks on the boundaries of the $\mathcal{G}^{*}(f )$-faces (i.e., the {\it basins of attraction} of the zeros, say $[a_{i}], i=1, \cdots , r,$ for $f$ on $T$ as attractors for the Newton flow $\overline{\overline{\mathcal{N}} }(f)$).
\begin{remark}
\label{R6.8N}
Note that in these facial walks the same vertex may occur more than once. However, by Lemma \ref{L6.2N}, 
a vertex in a facial walk cannot be adjacent to itself.
\end{remark}

\noindent
\underline{The orientations of $\mathcal{G}(f)$ and $\mathcal{G}^{*}(f )$}: \\
We endow (the faces of) $\mathcal{G}(f )$ with a coherent orientation as follows:

\noindent
For each facial walk we demand that the (constant) values of arg $f(z)$ on consecutive edges form an increasing sequence. This is imposed by the anti-clockwise ordering of the $\mathcal{G}(f )$-edges around a common vertex, which on its turn induces clockwise orientations of the $\mathcal{G}^{*}(f )$-edges incident to a given vertex. This leads to an orientation of (the facial walks on) $\mathcal{G}^{*}(f )$ which is opposite to the  orientation of $\mathcal{G}(f )$ as chosen before; see Fig.\ref{Figure13}. 
From now, on we assume that all graphs $\mathcal{G}(f )$ and $\mathcal{G}^{*}(f ), f \in \tilde{E}_{r},$  are oriented in this way: $\mathcal{G}(f )$
always 
clockwise; $\mathcal{G}^{*}(f )$
always anti-clockwise. By $- \mathcal{G}(f )$ we mean $\mathcal{G}(f )$ with anti-clockwise orientation and by $-\mathcal{G}^{*}(f )$ the clockwise oriented graph $\mathcal{G}^{*}(f )$.
\begin{lemma}
\label{L6.8N}
The (multi)graphs $\mathcal{G}(f )$ and  $\mathcal{G}^{*}(f )$
are {\it connected} and {\it cellularly embedded}\footnote{\label{FTN13}i.e. each face is homeomorphic to an {\it open} disk in $\mathbb{R}^{2}$.}.
\end{lemma}
\begin{proof}
We focus on $\mathcal{G}(f )$ and follow the treatise \cite{MoTh} closely.
Consider the $r$ facial walks $w_{b_{j}}$ and put $l_{j}=$ length $w_{b_{j}}$. Consider for each $w_{b_{j}}$ a so-called {\it facial polygon}, i.e. a polygon in the plane with $l_{j}$ sides labelled by the edges 
of $\partial F_{b_{j}}(f)$(taking the orientation of $w_{b_{j}}$  into account)
, so that each polygon is disjoint from the other polygons. Now we take all facial polygons. 
Each $\mathcal{G}(f )$-edge occurs precisely once in two different facial walks and this determines orientations of the sides of the polygons.
By identifying each side with its mate, we construct (cf. \cite{MoTh}) an orientable, connected surface $S$ , homeomorphic to $T$, and -in $S$- a 
2-cell embedded graph, which is -up to an isomorphism-equal to $\mathcal{G}(f )$. 
By Euler's formula for graphs on $T$ (cf. \cite{Gib}),
$\mathcal{G}(f )$ is connected and orientable 
as well. Finally, we note that a 2-cell embedding is always cellular (cf. \cite{MoTh}). 
\end{proof}

The {\it abstract directed graph}, underlying $\mathcal{G}(f) \wedge \mathcal{G}^{*}(f )$, will be denoted by $\mathbb{P}(f )$, where the directions are induced by the orientations of the (un)stable separatrices at $\overline{\overline{\mathcal{N}} }(f)$-saddles. Each canonical region is represented by a quadruple of directed edges in $\mathbb{P}(f )$, and is associated with precisely one pole, one zero (in opposite position) and two critical  points for $f$ on $T$. Following Peixoto (\cite{Peix1}, \cite{Peix2}), such a quadruple is called a {\it distinguished set} (of Type 1). The graph $\mathbb{P}(f )$), together with the collection of all distinguished sets is denoted by $\mathbb{P}^{d}(f )$. 
We say ``$\mathbb{P}^{d}(f )$ is realized by the distinguished graph of $\overline{\overline{\mathcal{N}} }(f)$ on $T$''.
\\ 

\noindent
We need a classical result due to Peixoto (cf. \cite{Peix2}) on structurally, $C^{1}$-stable vector fields 
on 2-dimensional compact manifolds. In the context of our elliptic Newton flows this yields (together with Lemma \ref{L1.1}
): if $ f, h \in \tilde{E}_{r}$, then:
\begin{equation}
\label{vgl22}
\overline{\overline{\mathcal{N}} }(f) \sim \overline{\overline{\mathcal{N}} }(h) \;\;\Leftrightarrow \;\;
\mathbb{P}^{d}(f ) \sim \mathbb{P}^{d}(h).
\end{equation}
Here, $\sim$ in the l.h.s stands for ``conjugacy'' and $\sim$ in the r.h.s. for {\it isomorphism between} $ \mathbb{P}^{d}(f ) $ and $ \mathbb{P}^{d}(h ) $ (as directed abstract graphs), preserving the distinguished sets and respecting
the cyclic ordering (induced by the embedding in $T$) of the distinguished sets around a common vertex.

\begin{theorem} $($Classification of structurally stable elliptic Newton flows by graphs$)$\\
\label{T6.10N}
Let $\overline{\overline{\mathcal{N}} }(f)$ and $\overline{\overline{\mathcal{N}} }(h)$ be structurally stable (thus
$f, h\in \tilde{E}_{r}),$ then:
\begin{equation*}
\overline{\overline{\mathcal{N}} }(f) \sim \overline{\overline{\mathcal{N}} }(h) \Leftrightarrow \mathcal{G}(f) \sim \mathcal{G}(h) 
(\text{and thus also } \mathcal{G}^{*}(f) \sim \mathcal{G}^{*}(h)),
\end{equation*}
where $\sim$ in the r.h.s. stands for equivalency between the oriented
graphs $($i.e., an isomorphism respecting their orientations$)$.
\end{theorem}
\begin{proof}
Apply (\ref{vgl22}) to $\overline{\overline{\mathcal{N}} }(f)$ and $\overline{\overline{\mathcal{N}} }(h)$.
\end{proof}

\noindent
Graph $\mathcal{G}(\frac{1}{f})$ is also well-defined (with as faces $F_{a_{i}}(\frac{1}{f})$) and associated with the structurally stable flow $\overline{\overline{\mathcal{N}} }(\frac{1}{f})$
(=$- \overline{\overline{\mathcal{N}} }(f)$).
The flow
$\overline{\overline{\mathcal{N}} }(\frac{1}{f})$ is the dual version of $\overline{\overline{\mathcal{N}} }(f)$, i.e., $\overline{\overline{\mathcal{N}} }(\frac{1}{f})$ is obtained from $\overline{\overline{\mathcal{N}} }(f)$ by reversing the orientations of the trajectories of the latter flow, thereby changing repellers into attractors and vice versa.
Clearly, $\mathcal{G}(\frac{1}{f})$ and $\mathcal{G}^{*}(f)$ coincide, be it with opposite orientations, i.e., $\mathcal{G}(\frac{1}{f})=-\mathcal{G}^{*}(f)$, where, due to our convention on orientations, $\mathcal{G}(\frac{1}{f})$ is clockwise oriented. 
Also, we have: $\mathcal{G}(f)=-\mathcal{G}^{*}(\frac{1}{f})$.

Note that, in general, $\overline{\overline{\mathcal{N}} }(f)$ and $\overline{\overline{\mathcal{N}} }(\frac{1}{f})$ are {\it not }conjugate. In the special case where $\overline{\overline{\mathcal{N}} }(f) \sim\overline{\overline{\mathcal{N}} }(\frac{1}{f})$ we call these flows {\it self dual}, 
 and we have: (Theorem \ref{T6.10N})
\begin{equation*}.
\overline{\overline{\mathcal{N}} }(f) \sim \overline{\overline{\mathcal{N}} }(\frac{1}{f}) \Leftrightarrow \mathcal{G}(f) \sim \mathcal{G}(\frac{1}{f})  \Leftrightarrow \mathcal{G}(f) \sim -\mathcal{G}^{*}(f)
\end{equation*}
If $ \mathcal{G}(f)  \sim - \mathcal{G}^{*}(f)$
holds, we call $ \mathcal{G}(f)$ and $ \mathcal{G}^*(f)$ {\it self dual}.\\

\begin{remark} 
\label{NT6.12} $($On the classification under conjugacy and duality$)$\\
Conjugate flows are considered as equal. Although, in general, $\overline{\overline{\mathcal{N}} }(f) $ and 
$\overline{\overline{\mathcal{N}} }(\frac{1}{f})$ are not conjugate, it is reasonable to consider also these flows, being related by a trivial (but orientation reversing) identity, as "equal". See our paper \cite{HT3}.
\end{remark}

\begin{remark}
\label{NR6.11}
(On self-duality)\\
If $\overline{\overline{\mathcal{N}} }(f) $ is self dual and conjugate with $\overline{\overline{\mathcal{N}} }(h) $, then $\overline{\overline{\mathcal{N}} }(h) $ is also self dual.
\end{remark}

\begin{corollary}
\label{C6.11N}
Any two structurally stable $2^{nd}$ order elliptic Newton flows are conjugate. In particular, these flows are self dual.
\end{corollary}
\begin{proof}
Let $\overline{\overline{\mathcal{N}} }(f), f  \in \tilde{E}_{2},$ be chosen arbitrarily. By Corollary \ref{C6.3},
the two faces of $\mathcal{G}(f)$ share their boundaries. So, the common facial walk $w_{f}$ of these faces is built up from the four $\mathcal{G}(f)$-edges and the two $\mathcal{G}(f)$-vertices (each appearing twice but not consecutive!). 
Hence, compare the construction in the proof of Lemma \ref{L6.8N} and see Fig. \ref{Figure14}, $\mathcal{G}(f)$ is determined by the anti-clockwise
oriented walk $w_{f}$.
The same holds for any other flow $\overline{\overline{\mathcal{N}} }(h)$ with facial walk $w_{h}, h  \in \tilde{E}_{2}.$ 
Apparently, $w_{f}$ and $w_{h}$ maybe considered as equal, under a suitably chosen relabeling of their vertices and edges. Hence $\mathcal{G}(f) \sim \mathcal{G}(h)$ and thus $\overline{\overline{\mathcal{N}} }(f) \sim \overline{\overline{\mathcal{N}} }(h)$. In particular, put $h=\frac{1}{f}$, then we find $\mathcal{G}(f) \!\sim\! \mathcal{G}(\frac{1}{f})$, compare Fig. \ref{Figure14} and Remark \ref{NR6.11}.
\end{proof}

\begin{figure}[h]
\begin{center}
\includegraphics[scale=1.3]{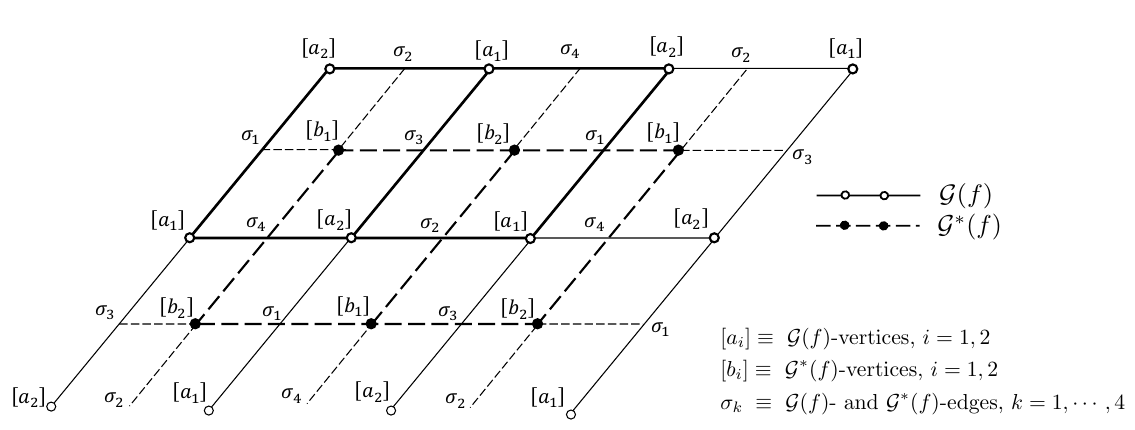}
\caption{\label{Figure14}The graphs $\mathcal{G}(f)$ 
and $\mathcal{G}^{*}(f)$ 
, $f \! \in \!\tilde{E}_{2}$.}
\end{center}
\end{figure}

\begin{remark}
\label{R6.9}
For basically the same proof of Corollary \ref{C6.11N} 
, see \cite{THBS}.
\end{remark}
\begin{remark}
\label{R6.10}
 The flow $\overline{\overline{\mathcal{N}} }$(sn), in the non-rectangular case (cf. Fig.\ref{Figure3N} ) exhibits an example of a  $2^{nd}$ order structurally stable elliptic Newton flow. By Corollary \ref{C6.11N}, this is the only possibility (up to conjugacy) for a flow in $\tilde{N}_{2}$ . Note that the flow in Fig.\ref{Figure4N} (rectangular case) is not structurally stable (because of the saddle connections).
\end{remark}

\noindent
We proceed by introducing flows that are closely related to $\mathcal{N}(f)$, $\overline{\mathcal{N} }(f)$ and 
$\overline{\overline{\mathcal{N}} }(f)$: the so-called {\it rotated Newton flows}, 
\begin{definition}
\label{D6.15N}
 For $f \in E_{r}$, let $\mathcal{N}^{\perp}(f)$  be a dynamical system of the type
 \begin{equation*}
  \dfrac{dz}{dt} = \dfrac{-if (z)}{f^{'} (z)}.
\end{equation*}
\end{definition}
Apparently, $\mathcal{N}^{\perp}(f)(=i\mathcal{N}(f))$
is a complex analytic vector field outside the set $C(f)$ of critical points for $f$. As in Subsection 1.1
, we turn 
$\mathcal{N}^{\perp}(f)$
into a $C^{1}$-system on the whole plane
with -on $\mathbb{C} \backslash C(f)$-
the same phase portrait as $\mathcal{N}^{\perp}(f)$ by $\overline{\mathcal{N} }^{\perp}(f):=i \overline{\mathcal{N} }(f)$.
The function $f$, being elliptic, the system $\overline{\mathcal{N} }^{\perp}(f)$
can be interpreted as a $C^{1}$-flow on $T$ and as such it will be referred to as to $\overline{\overline{\mathcal{N}} }^{\perp}(f)$,
in particular: 
$$
\overline{\overline{\mathcal{N}} }^{\perp}(f) \text{ is of the class $C^{1}$, and } \overline{\overline{\mathcal{N}} }^{\perp}(\frac{1}{f})=-\overline{\overline{\mathcal{N}} }^{\perp}(f)
$$

\begin{lemma}
\label{L6.16N}
Let $z^{\perp}(t)$
be the (maximal)
$\mathcal{N}^{\perp}(f)$-trajectory through a non-equilibrium $\check{z}=z^{\perp}(0)$
, then: 
\begin{enumerate}
\item[1.] $f(z^{\perp}(t))=e^{-it}f(\check{z}) [$thus $|f(z^{\perp}(t))|$= constant $(\neq 0)]$.
\item[2.] A zero or pole for $f$ is a center for $\mathcal{N}^{\perp}(f) [$thus also for $\overline{\mathcal{N} }^{\perp}(f)$, $\overline{\overline{\mathcal{N}} }^{\perp}(f) ]$. 
\item[3.] A $k$-fold critical point for $f$ is a $k$-fold saddle for $\overline{\mathcal{N} }^{\perp}(f)$ 
$[$thus also for $\overline{\overline{\mathcal{N}} }^{\perp}(f)].$
\end{enumerate}
\end{lemma}
\begin{proof}
 Assertions 1., 3. : Use $\mathcal{N}^{\perp}(f)=i\mathcal{N}(f)$.
Note that outside $N(f) \cup P(f)$ the flow $\mathcal{N}^{\perp}(f)$ can be considered as the Newton flow for $h(z)=\exp (-i \log (f(z)))$. \\
For Assertion 2. : let $z_{0}$  be a zero or pole for $f$ with multiplicity $k$, thus an isolated zero for $\mathcal{N}^{\perp}(f)$. In a neighborhood of $z_{0}$, system $\mathcal{N}^{\perp}(f)$ is linearly approximated by:
 \begin{equation*}
  \dfrac{dz}{dt} = \dfrac{-i (z-z_{0})}{k}.
\end{equation*}
Thus $z_{0}$ is a non-degenerate equilibrium for $\mathcal{N}^{\perp}(f)$ with characteristic roots  $\pm \frac{i}{k}$. 
By the first assertion in the lemma, a regular integral curve through a point $\check{z}$ close to $z_{0}$, but $\neq z_{0}$, cannot end up at, or leave from $z_{0}$. Hence, this point is neither a focus, nor a centro-focus for $\mathcal{N}^{\perp}(f)$ (cf. \cite{ALGM})
and must be a center for $\mathcal{N}^{\perp}(f)$.
\end{proof}
In view of the above Assertion 1., a closed orbit for $\overline{\overline{\mathcal{N}} }^{\perp}(f)$ cannot be a limit cycle, and -by 2.- a separatrix $z^{\perp}(t)$ leaving a saddle $\sigma_{1}$, must approach a saddle $\sigma_{2}$. 
Moreover, this separatrix cannot connect $\sigma_{1}$ to itself, i.e. $\sigma_{1} \neq \sigma_{2}$. In fact, let $\sigma_{1}=\sigma_{2}$, since there holds that 
arg $h (z^{\perp}(t))=$ constant
$$
\lim_{t \downarrow 0} \text{ arg }h (z^{\perp}(t))=\text{ arg }h(\sigma_{1})     \text{ and also }     \lim_{t \uparrow 0} \text{ arg }h (z^{\perp}(t))=\text{ arg }h(\sigma_{1}),
$$
which is impossible, see Fig. \ref{Figure15N} and the Comment on Fig.\ref{Figure1}.
\begin{figure}[h]
\begin{center}
\includegraphics[scale=1.1]{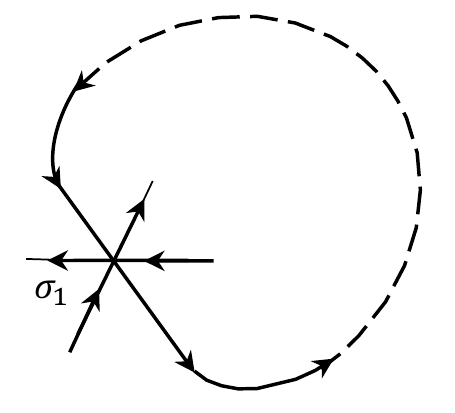}
\caption{\label{Figure15N}No ``self-connected'' $\overline{\overline{\mathcal{N}} }^{\perp}(h)$-saddles
; $\sigma_{1}$ is 2-fold; $h(z)=\exp (-i \log (f(z)))$.}
\end{center}
\end{figure}

Note that -when introducing rotated Newton flows- no additional restrictions were laid upon the function $f$. But now, we return to the case of {\it non-degenerate} functions $f$.
Then $\overline{\overline{\mathcal{N}} }^{\perp}(f)$ has 2$r$ simple saddles (corresponding to the critical points for $f$) with altogether $4r$ separatrices, connecting {\it different } saddles. So, we may introduce:

\newpage
\begin{definition}
\label{D6.17N}
The graph $\mathcal{G}^{\perp}(f), f \in \check{E}_{r},$
on the torus $T$ is given by:
\begin{itemize}
\item Vertices are the 2$r$ critical points for $f$ (as saddles for $\overline{\overline{\mathcal{N}} }^{\perp}(f)$) on $T$.
\item Edges are the 4$r$ separatrices at the critical points for $f$ (as $\overline{\overline{\mathcal{N}} }^{\perp}(f)$-saddles) on $T$.
\end{itemize}
\end{definition}
Since all zeros and poles for $f$ are centers for $\overline{\overline{\mathcal{N}} }^{\perp}(f)$, each $\mathcal{G }^{\perp}(f)$-face contains only one zero or one pole for $f$. Moreover, the graph $\mathcal{G}^{\perp}(f)$ is cellularly embedded. Hence, the graph $\mathcal{G}^{\perp}(f)$ has 2$r$ faces.

Let $c$ be an arbitrary, strictly positive real number and put $L_{c}=\{z \mid |f(z)|=c\}$.
\begin{lemma}
\label{L6.18N}
Then there holds:
\begin{enumerate}
\item[(1)] The level set $L_{c}$ is a regular curve in $\mathbb{R}^{2}$ (i.e., grad $|f(z)| \neq 0$ for all $z \in L_{c}$)
                             if and only if $L_{c}$ contains no critical points for $f$.
\item[(2)]The graph $\mathcal{G}^{\perp}(f), f \in \check{E}_{r},$ is connected. In particular, $f(z)$ admits the same
                              absolute value at all critical points $z$.
\end{enumerate}
\end{lemma}
\begin{proof}
(1): Use the Cauchy-Riemann equations.\\
(2): Apply Euler's formula for toroidal graphs. (cf. \cite{Gib})
\end{proof}
\noindent
We orient the edges of $\mathcal{G}^{\perp}(f)$  according to their orientation as  $\overline{\overline{\mathcal{N}} }^{\perp}(f)$-trajectories. 
Let $A_{i}$  and $B_{j}$ be open subsets of $\mathbb{C}$, corresponding to the (open) faces of $\mathcal{G}^{\perp}(f)$ that are determined by the zero $a_{i}$, respectively the pole $b_{j}$, for $f$.
Hence, the boundaries of $A_{i}$ are clockwise oriented, but those of $B_{j}$ anti-clockwise.  Since $\overline{\overline{\mathcal{N}} }^{\perp}(\frac{1}{f})= -\overline{\overline{\mathcal{N}} }^{\perp}(f)$ we have: reversing the orientations in $\mathcal{G}^{\perp}(f)$ turns this graph into $\mathcal{G}^{\perp}(\frac{1}{f})$ and thus, by Lemma \ref{L6.18N} (2):  $|f(z)|=1$ on $\mathcal{G}^{\perp}(f)$. See Fig. \ref{Figure16N}
for (parts of) the graphs $\mathcal{G}^{\perp}(f)$, $\mathcal{G}(f)$, $\mathcal{G}^{*}(f)$ and $\mathcal{G}^{\perp}(\frac{1}{f})$, $\mathcal{G}(\frac{1}{f})$, $\mathcal{G}^{*}(\frac{1}{f})$.
A canonical $\overline{\overline{\mathcal{N}} }(f)$-region, with $[a_{i}]$,$[b_{j}]$ in opposite position, and the saddles $\sigma, \sigma'$ consecutive w.r.t. the orientation of $A_{i}$(or $B_{j}$), will be denoted by $R_{ij}(\sigma, \sigma')$
and is contained in $F_{a_{i}}(\frac{1}{f}) \cap F_{b_{j}}(f)$.
Note that, in general, this intersection contains more canonical regions of type $R_{ij}(\cdot, \cdot)$. But even so, these regions  are separated by canonical regions, not of this type; compare Remark \ref{R6.8N}.
In view of Subsection 1.1
and Lemma \ref{L6.16N} (1): Under $f$
the 
net of $\overline{\overline{\mathcal{N}}}(f)$- and $\overline{\overline{\mathcal{N}}}(f)^{\perp}$-trajectories on $R_{ij}(\sigma, \sigma')$ is homeomorphically 
mapped onto 
a polar net in a sector of the $u+iv$-plane ($u$=Re($f$), $v$=Im($f$)), namely
$$
s_{i,j}(\sigma, \sigma')=\{ (u,v) \mid 0 < u^{2}+v^{2} < \infty , \text{arg}f(\sigma) < \arctan (\frac{v}{u}) < \text{arg}f(\sigma')\}.
$$
Analogously, $\frac{1}{f}$ maps the net of $\overline{\overline{\mathcal{N}}}(f)$- and $\overline{\overline{\mathcal{N}}}(f)^{\perp}$-trajectories on $R_{ji}(\sigma, \sigma')$ 
onto 
a polar net in a sector of the $U+iV$-plane ($U$=Re($\frac{1}{f}$), $V$=Im($\frac{1}{f}$)), namely
$$
S_{i,j}(\sigma, \sigma')=\{ (U,V) \mid 0 < U^{2}+V^{2} < \infty , -\text{arg}f(\sigma) < \arctan (\frac{V}{U}) < -\text{arg}f(\sigma')\}.
$$
So, the polar nets on $s_{i,j}(\sigma, \sigma')$ and $S_{i,j}(\sigma, \sigma')$ correspond under the inversion\footnote{\label{FN26} Here we use that in a canonical region the angles at the zero and the pole are equal.}: $$U=\frac{u}{u^{2}+v^{2}}, V=\frac{v}{u^{2}+v^{2}}.$$

\begin{figure}[h]
\begin{center}
\includegraphics[scale=1.5]{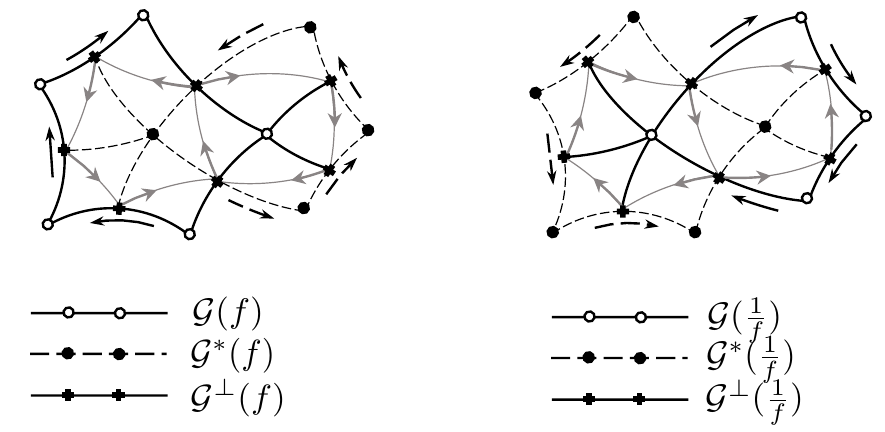}
\caption{\label{Figure16N}The graphs $\mathcal{G}(\cdot)$, $\mathcal{G}^{*}(\cdot )$ and $\mathcal{G}^{\perp}(\cdot)$ for $f$ and $\frac{1}{f}$ in $\tilde{E}_{r}$.}
\end{center}
\end{figure}

\noindent
Next we turn to the relationship between Newton flows and steady streams.

\noindent
\begin{remark} {\it Newton flows as steady streams.}\\
\label{R6.20}
\noindent
For $f \in \check{E}_{r}$, we consider the planar {\it steady stream} (\cite{M1}) with complex potential $w(z)=-\log f(z)$
, potential function $\Phi(x,y)=-\log |f(z)|$ and stream function $\psi(x,y)=- \text{ arg}f(z)$, where $x=\text{Re}(z), y=\text{Im}(z).$ Then the {\it equipotential lines} are given by $ - \log |f(z)|=$ constant, the {\it stream lines} by  $-\text{ arg}f(z)=$ constant and the {\it velocity field }$V(z)(=\text{ grad} \Phi)$ by the complex conjugate of $w'(z)$, i.e.
$$
V(z)=\frac{|w'(z)|^{2}}{w'(z)}=\frac{-|w'(z)|^{2}f(z)}{f'(z)}(=|w'(z)|^{2} \mathcal{N}(f)).
$$
Moreover, the zeros (poles) for $f$ are just the sinks (sources) of strength 1, whereas the critical points for $f$ are the 1-fold stagnation points  of the stream, compare also \cite{HT1}. So, the ``orthogonal net of the stream- and equipotential-lines'' of the planar steady stream is a combination of the phase portraits of $\overline{\mathcal{N} }(f)$ and $\overline{\mathcal{N} }^{\perp}(f)$, see Fig. \ref{Figure15NN}.
%17. 

\begin{figure}[h]
\begin{center}
\includegraphics[scale=1.2]{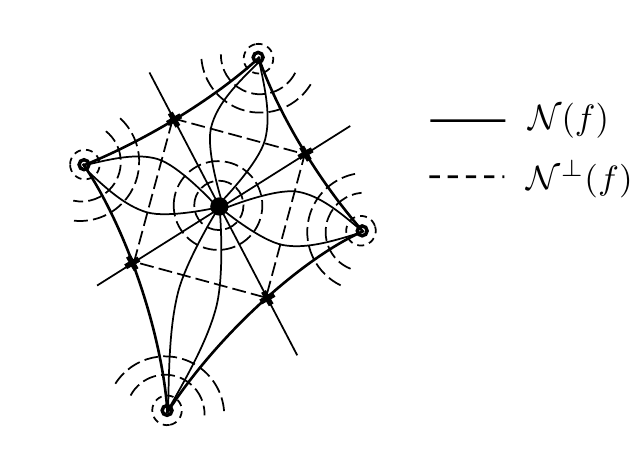}
\caption{\label{Figure15NN}The steady stream $w(z)=-\log f(z), f \in \check{E}_{r}.$}
\end{center}
\end{figure}

Hence we may interprete the pair ($\overline{\overline{\mathcal{N}} }(f)$, $\overline{\overline{\mathcal{N}} }^{\perp}(f)$) as a toroidal desingularized version of our planar steady stream.
\noindent
Finally, we clarify the ``steady stream character'' of the structurally stable elliptic Newton flows from the point of view of the Riemann surface $T$. 

\noindent 
Firstly, we note that the polar net on open (!) sectors as $s_{i,j}(\sigma, \sigma')$ and $S_{i,j}(\sigma, \sigma')$ are just the stream and equipotential lines of the steady stream with complex potential $-\log(u+iv)$, resp. $-\log(U+iV)$. In particular , these stream and equipotential lines exhibit the phase portraits of resp. the flows $\mathcal{N}(u+iv)(=-(u+iv))$, $\mathcal{N}(u+iv)^{\perp}(=-i(u+iv))$, and $\mathcal{N}(U+iV)(=-(U+iV))$, $\mathcal{N}(U+iV)^{\perp}(=-i(U+iV))$        
on $s_{i,j}(\sigma, \sigma')$ and $S_{i,j}(\sigma, \sigma')$ respectively. Deleting from $T$ all zeros, poles and critical points for $f$, we obtain ``the reduced torus''
$\check{T}$: an open submanifold of $T$.

\noindent
Now the collection
$$
\{ F_{a_{i}}(\frac{1}{f}) \backslash [a_{i}], F_{b_{j}}(f) \backslash [b_{j}]; i,j= 1, \cdots r \}
$$
exhibits a covering of $\check{T}$ with open neighborhoods.
Apparently, only in the case of pairs $(F_{a_{i}}(\frac{1}{f}) \backslash [a_{i}], F_{b_{j}}(f) \backslash [b_{j}])$
a non-empty intersection is possible. Even so, the intersection $$F_{a_{i}}(\frac{1}{f}) \backslash [a_{i}] \cap F_{b_{j}}(f) \backslash [b_{j}]$$
consists of the disjoint union of sets of the type $R_{ij}(\cdot, \cdot)$, say $R_{ij}^{1}, \cdots ,R_{ij}^{s}$.
 (Note that $[a_{i}]$ occurs in $w_{b_{j}}$ as many times as $[b_{j}]$ occurs in $w_{a_{i}}$). This turns our covering into an atlas for $\check{T}$ with smooth (even complex analytic) coordinate transformations, induced by the inversion $(u+iv)) \leftrightarrow (1/(u+iv))=U+iV$.  
With aid of this atlas, we may interprete  $\overline{\overline{\mathcal{N}} }(f)$ and $\overline{\overline{\mathcal{N}} }^{\perp}(f)$ on each canonical region  as the Ò pull backÓ of theÓ most simple\footnote{\label{FN27}On the sectors $s_{i,j}(\sigma, \sigma')$ resp. $S_{i,j}(\sigma, \sigma')$ the flows $\mathcal{N}(u+iv)$, resp. $\mathcal{N}(U+iV)$ are parts of North-South flows. (cf. 
\cite{HT1})}
planar flows 
$\mathcal{N}(u+iv), \mathcal{N}^{\perp}(u+iv)$, and $\mathcal{N}(U+iV), \mathcal{N}^{\perp}(U+iV)$  
on the various sectors $s_{i,j}(\cdot , \cdot)$ and $S_{i,j}(\cdot, \cdot)$
respectively. Glueing the canonical regions $R_{ij}(\cdot, \cdot)$ along the $\overline{\overline{\mathcal{N}} }(f)$-trajectories in their common boundaries, we  obtain the restrictions to $\check{T}$ of our original (rotated) Newton flows. In particular, the flows $\mathcal{N}(u+iv)(=-(u+iv))$ and $\mathcal{N}(U+iV)(=-(U+iV))$ lead to an analytic function on $\check{T}$, namely the restriction $f | \check{T}$, with as isolated singularities the zeros, poles and critical points for $f$. 
By continuous extension to this singularities, we find the original flows $\overline{\overline{\mathcal{N}} }(f)$ and $\overline{\overline{\mathcal{N}} }^{\perp}(f)$. 
For an illustration, see Fig. \ref{Figure18N}, \ref{Figure19N}.
\end{remark}

\newpage

\begin{figure}[h!]
\begin{center}
\includegraphics[scale=0.6]{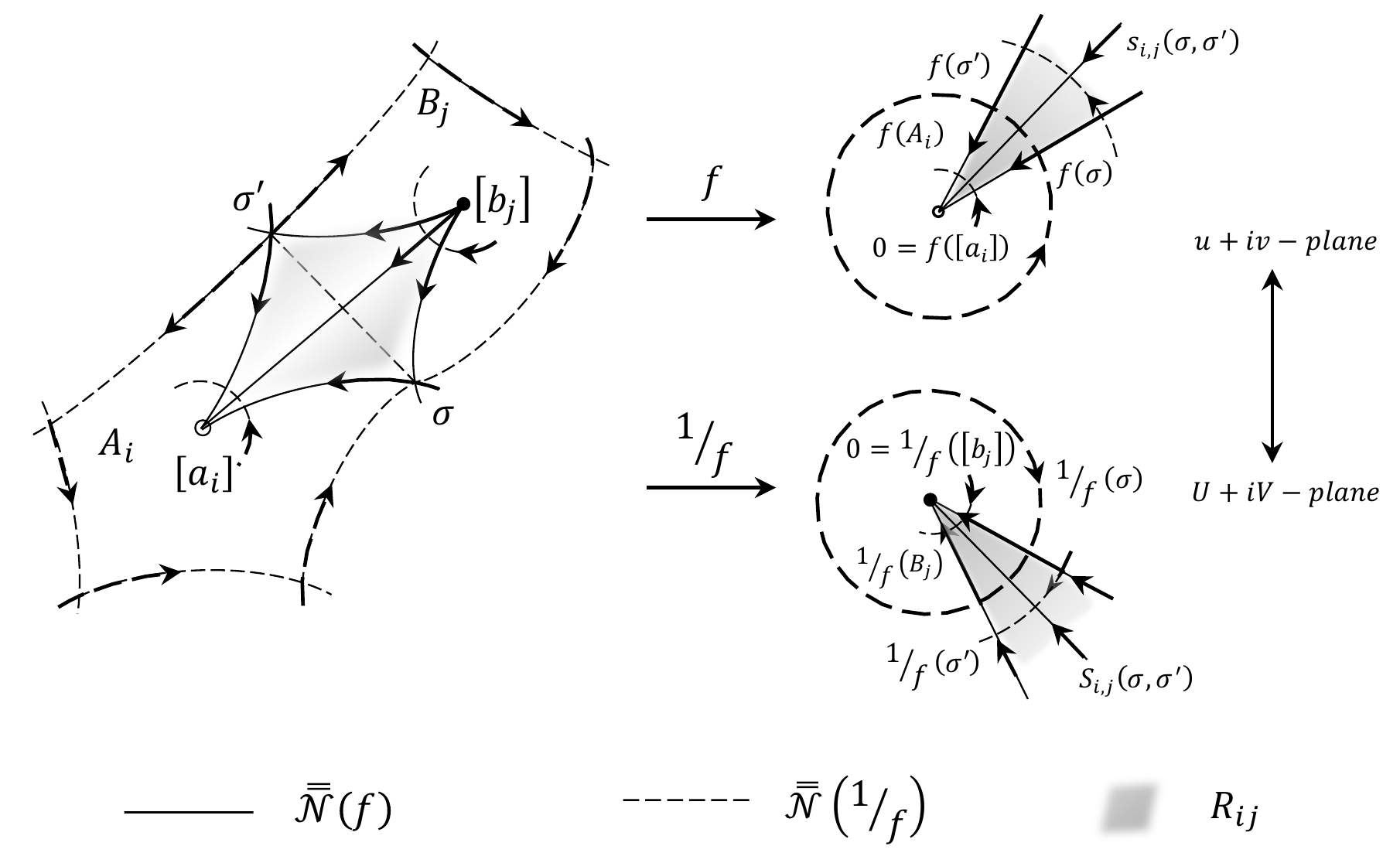}
\caption{\label{Figure18N}
The canonical regions $R_{ij}$, and the sectors $s_{i,j}(\sigma, \sigma')$ and $S_{i,j}(\sigma, \sigma')$.}
\end{center}
\end{figure}

\begin{figure}[h!]
\begin{center}
\includegraphics[scale=0.5]{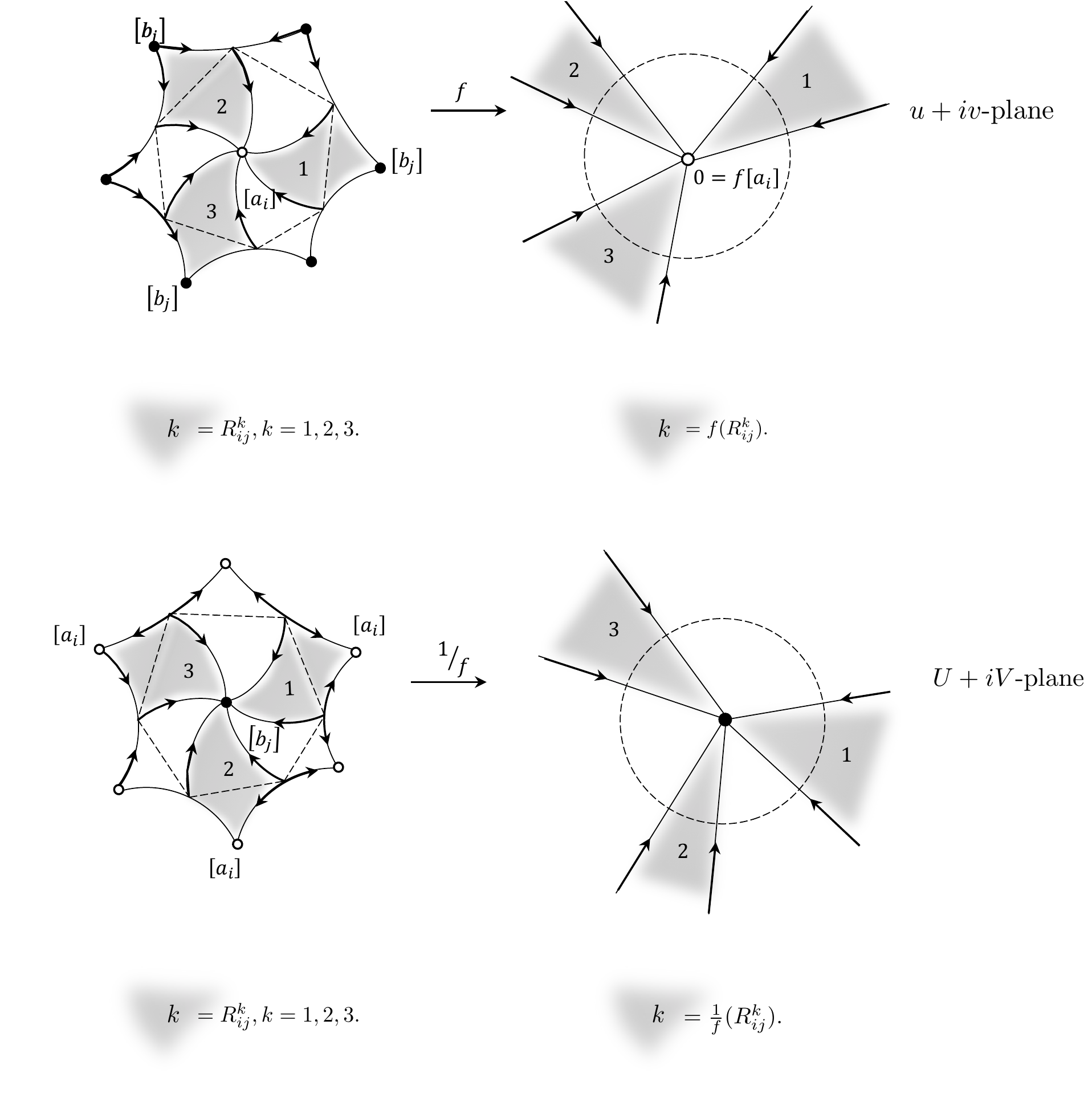}
\caption{\label{Figure19N}
$F_{a_{i}}(\frac{1}{f}) \backslash [a_{i}] \cap F_{b_{j}}(f) \backslash [b_{j}] $ and its images under $f$ and $\frac{1}{f}$ ($s=3$).}
\end{center}
\end{figure}

\newpage
\section{Newton graphs}
\label{Nsec7}

\noindent
Throughout this section, the connected graph $\mathcal{G}$ is a cellular embedding in $T$, seen as a compact, orientable, Hausdorff topological space, of 
an abstract connected {\it multigraph}  {\bf G} (i.e., no loops)
with $r$ vertices, 2$r$ edges
($r \geqslant 2$); $r=${\it order} $\mathcal{G}$.\\

\noindent
The forthcoming analysis strongly relies on some concepts from classical graph theory on surfaces, which- in order to fix terminology-will be briefly reviewed\footnote{\label{7FN1} Again we follow the treatise \cite{MoTh} closely. Note however, that in \cite{MoTh} a multigraph may exhibit loops, whereas in our case this possibility for $\mathcal{G}$ is ruled out.}:\\

\noindent
{\large{\bf 3.1. Cellularity; geometric duals}}\\

\noindent
Since $\mathcal{G}$ is cellularly embedded, we may consider (cf. \cite{MoTh}) the {\it rotation system} $\Pi$ for $\mathcal{G}$:
$$
\Pi=\{ \pi_{v} \!\mid \! \text{all vertices $v$ in {\bf G}} \},
$$
where the local {\it rotation $\pi_{v}$} at $v$ is the cyclic permutation of the edges incident with $v$ such that $\pi_{v}(e)$ is the successor of $e$ in the anti-clockwise ordering around $v$.\\

\noindent
If $e(=v'v'')$ stands for an edge, with end vertices $v'$ and $v''$, we define a $\Pi$-{\it walk} (facial walk\footnote{\label{7FN2}Compare the facial walk $w_{b_{j}}$ in Section \ref{sec6}.}), say $w$, on $\mathcal{G}$
as follows:\\
\\ 
\underline{The face traversal procedure.}
\\
Consider an edge $e_{1}=(v_{1}v_{2})$ and the closed walk\footnote{\label{7FN3}We shall not distinguish between $w$ and its cyclic shifts.}
$w=v_{1}e_{1}v_{2}e_{2}v_{3} \cdots v_{k}e_{k}v_{1}$, which is determined by the requirement that, for $i=1, \cdots , \ell,$ we have $\pi_{v_{i+1}}(e_{i})=e_{i+1}$, where $e_{\ell +1}=e_{1}$ and $\ell$ is minimal.\\
\\ 
\noindent
Apparently, such ``minimal'' $\ell$ exists since ${\bf G}$ is finite. Note that each edge occurs {\it either once} in two different $\Pi$-walks, or {\it twice} (with opposites orientations) in {\it only one $\Pi$-walk}; in particular, the first edge in the same direction which is repeated when traversing $w$, is $e_{1}$. As in the proof of Lemma \ref{L6.8N}, these $\Pi$-walks can be used to construct (patching the facial polygons along identically labelled sides) a
surface $S$ and in $S$, a so-called {\it 2-cell} 
embedded graph  
with faces 
determined by the facial polygons. By Euler's formula (cf. \cite{Gib}) there are $r$ facial walks. So, $S$ is homeomorphic to $T$ and the 2-cell embedded graph is isomorphic to $\mathcal{G}$. By the {\it Heffter-Edmonds-Ringel rotation principle}, the graph $\mathcal{G}$ is
uniquely determined up to isomorphism
by its rotation system.
We say: $\mathcal{G}$ is generated by $\Pi$.
\\
\noindent
From now on, we suppress the role of the underlying abstract graph ${\bf G}$ and will not distinguish between the vertices of $\mathcal{G}$ and those of ${\bf G}$. Occasionally, $\mathcal{G}$ will be referred to as to the pair ($\mathcal{G}$, $\Pi$). The $\mathcal{G}$-faces (as well as the corresponding facial polygons) are denoted by $F_{j}$; their boundaries (as well as the corresponding $\Pi$-walks) by $\partial F_{j}, j=1, \cdots ,r.$ We denote the sets of all vertices, edges and faces of $\mathcal{G}$ by $V(\mathcal{G}), E(\mathcal{G})$ and $F(\mathcal{G})$ respectively. \\
\\
\noindent
The embedding of $\mathcal{G}$ into the orientable surface $T$ induces an anti-clockwise orientation on the edges around each vertex $v$. In the sequel we assume that the local rotations $\pi_{v}$ are endowed with this orientation (so that the inverse permutation $\pi_{v}^{-1}$ are clockwise).\\
\\
\noindent
Given a cellularly embedded toroidal ($\mathcal{G}$, $\Pi$), the abstract graph ${\bf G}^{*}$ is defined as follows:
\begin{itemize}
\item The $r$ vertices $\{ v^{*}\}$ are represented by the $\Pi$-walks in $\mathcal{G}$,
\item Two vertices are connected by an edge $e^{*}$ iff the representing $\Pi$-walks share an edge $e$.
\end{itemize}
Hence, 
between the $\mathcal{G}$-edges and ${\bf G}^{*}$-edges, there is a bijective correspondence: $e \leftrightarrow e^{*}$.\\
In particular, 
${\bf G}^{*}$ has $2r$ edges, and an edge $e^{*}$ is a loop\footnote{\label{7FN4} In contradistinction to our assumption on ${\bf G}$, the graph ${\bf G}^{*}$ may admit loops.} iff $e$ shows up twice in a $\Pi$-walk of $\mathcal{G}$.\\
\\
\noindent
The graph ${\bf G}^{*}$ admits a 2-cell embedding in $T$: the {\it (geometric) dual } ($\mathcal{G}^{*}$, $\Pi^{*}$). In fact, if the vertex $v^{*}$ in ($\mathcal{G}^{*}$, $\Pi^{*}$) is represented\footnote{\label{7FN5} We say: $v^{*}$ is  ``'located' in the $\mathcal{G}$-face, determined by the $\Pi$-walk ($e_{1}-\cdots -e_{\ell}$).} by the $\Pi$-walk ($e_{1}-\cdots -e_{l}$), then the cyclic permutation on the ${\bf G}^{*}$-edges incident with $v^{*}$, say $\pi_{v^{*}}^{*}$, is defined by $\pi_{v^{*}}^{*}=(e_{1}^{*}-\cdots -e_{\ell}^{*})$. A $\Pi^{*}$-walk of length $\ell'$ corresponds to precisely one $\mathcal{G}$-vertex  of degree $\ell'$: compare Fig.\ref{Figure13}, where $\mathcal{G}=\mathcal{G}(f)$ and $\mathcal{G}^{*}$=$\mathcal{G}^{*}(f)$. The anti-clockwise orientation of the local rotation systems in $\mathcal{G}$
induces a clockwise orientation on the $\Pi$-walk 
in $\mathcal{G}$
and thus a clockwise orientation on the rotation systems in $\mathcal{G}^{*}$
This results -by the face traversal procedure-into an anti-clockwise orientation on the $\Pi^{*}$-walks in $\mathcal{G}^{*}$. \\
\noindent
By $-\mathcal{G}$ ($-\mathcal{G}^{*}$) we mean  $\mathcal{G}$ ($\mathcal{G}^{*}$) with the anti-clockwise (clockwise)orientation; compare $-\mathcal{G}(f)$ and $-\mathcal{G}^{*}(f)$ in Section 2.
It follows that $\text{($\mathcal{G}^{*}$, $\Pi^{*}$)}^{*}$=($\mathcal{G}$, $\Pi$).
\noindent
Note that two cellularly embedded graphs in $T$ are isomorphic, then also their duals.\\

\noindent
{\large {\bf 3.2. The E(Euler)-property}}\\

\noindent
In contradistinction to the case of facial walks in $\mathcal{G}(f), f \in \tilde{E}_{r},$ see Lemma \ref{L6.7N}, a $\Pi$-walk in $\mathcal{G}$ is-in general- {\it not} an Euler-trail. So, we need an additional condition:

\begin{definition}
\label{ND7.2}
($\mathcal{G}$, $\Pi$) has the {\it E(Euler)-property} if every $\Pi$-walk is Eulerian.
\end{definition}

For an example of a second order graph ($\mathcal{G}$, $\Pi$) that has the {\it E-property}, see Fig. \ref{Figure15}-(i). This it not so for the third order graphs ($\mathcal{G}$, $\Pi$) in Fig. \ref{Figure15}-(ii), (iii), whereas the graph in Fig. \ref{Figure15} (iv) does not even fulfil the initial conditions laid upon $\mathcal{G}$ (because there are 3 vertices and only 5 edges). Note, however, that also in the latter case the Euler Characteristic vanishes, so that this multigraph is toroidal as well.

\begin{figure}[h]
\begin{center}
\includegraphics[scale=0.6]{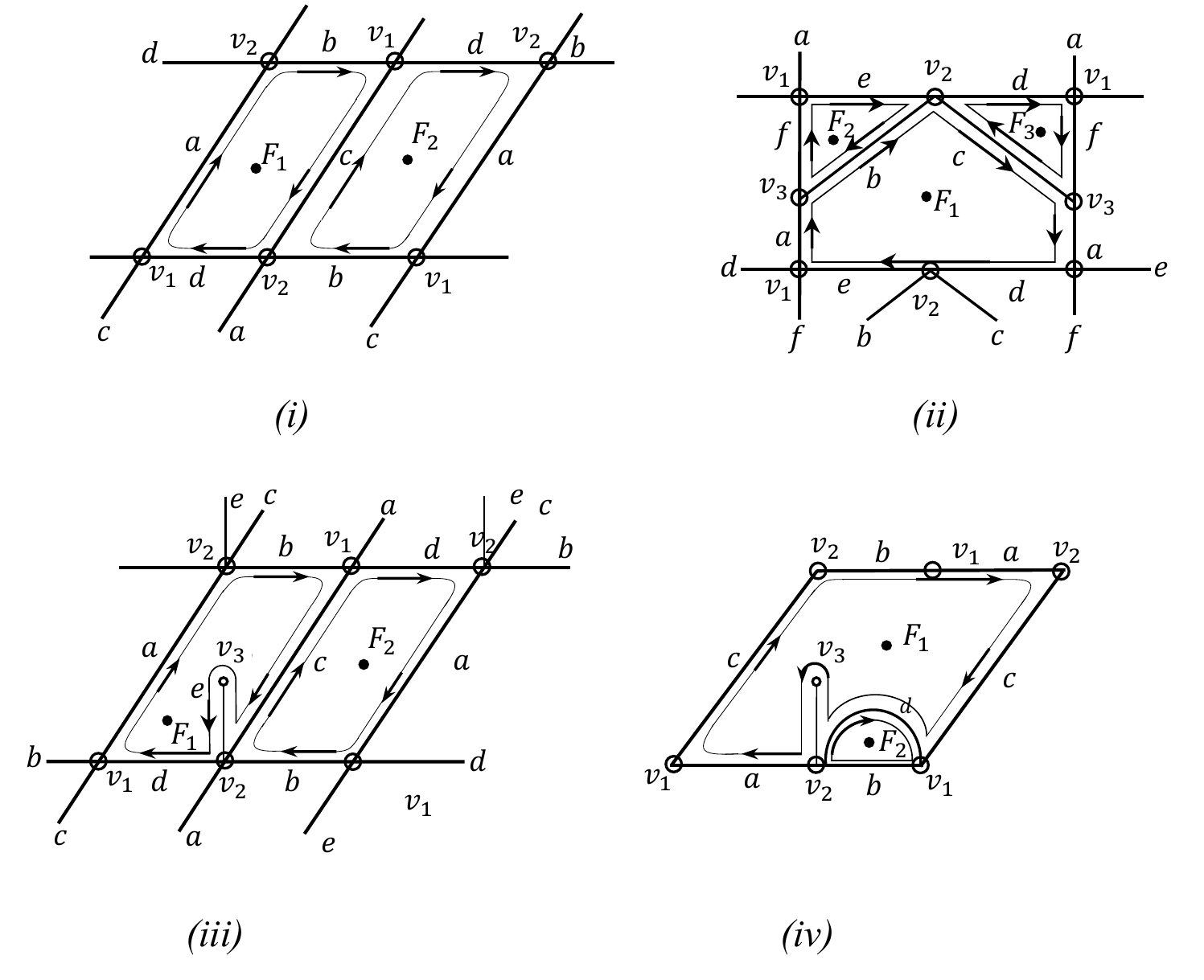}
\caption{\label{Figure15}Four multigraphs, cellularly embedded in $T$} 
\end{center}
\end{figure}

\begin{lemma}
\label{NL7.3}
If $(\mathcal{G}$, $\Pi )$ has the {\it E-property}, then this is also true for $(\mathcal{G}^{*}$, $\Pi^{*})$.
\end{lemma}
\begin{proof}
Recall that the conditions ``{\it E-property} holds for $\mathcal{G}$'' and ``non-occurrence of loops in $\mathcal{G}^{*}$'' are equivalent and 
apply $(\mathcal{G}^{*})^{*}$=$\mathcal{G}$.
\end{proof}

\noindent
From now, on we assume that both $\mathcal{G}$
and $\mathcal{G}^{*}$
are multigraphs and fulfil the {\it E-property}. In particular, 
each 
edge in these graphs is
adjacent to {\it two different} 
faces.\\

\noindent
Let $v$ be an arbitrary vertex in $\mathcal{G}$, contained in the boundary $\partial F$ of a face $F$ and  $e_{1}ve_{2}$  a subwalk of the $\Pi$-walk $w_{F}$. The {\it different} edges  $e_{1}$, $e_{2}$ are consecutive w.r.t. the ({\it clockwise}) orientation of $w_{F}$. The {\it facial local sector} of $F$ at $v$, spanned by the ordered pair ($e_{1}$, $e_{2}$), is referred to as to a {\it $F$-sector at} $v$. Note that if $v$ occurs more than once in $w_{F}$ , two $F$-sectors at $v$ cannot share an edge (because in that case the common edge would show up twice in $w_{F}$). 
Hence,  $F$-sectors at $v$ must be separated by facial sectors at $v$ that do not belong to $F$. So, if $e_{1}ve_{2}$  and $e^{'}_{1}ve^{'}_{2}$  
are subwalks of $w_{F}$, spanning two facial $F$-sectors at $v$, then $e_{1}$, $e_{2}$, $e^{'}_{1}$ and  $e^{'}_{2}$ must be different. Thus each vertex in $\partial F$ has even degree.

Apparently, the number of all facial sectors at $v$ equals the degree of $v$, and in $\mathcal{G}$ there are altogether $\delta_{1}+ \cdots + \delta_{r}(=4r)$
facial sectors, where the $\delta_{i}$'s
stand for the degrees of the vertices in $\mathcal{G}$. \\

\noindent
Similarly, there are $\delta^{*}_{1}+ \cdots + \delta^{*}_{r}(=4r)$
facial sectors in $\mathcal{G}^{*}$ with the  $\delta^{*}_{j}$'s the degrees of the $\mathcal{G}^{*}$-vertices.\\

\noindent
We write $F=F_{v^{*}}$, where $v^{*}$ is the $\mathcal{G}^{*}$-vertex defined by $F$.
So, $w_{F}=w_{F_{v^{*}}}$. Analogously, $F^{*}_{v}$ stands for the $\mathcal{G}^{*}$-face determined by $v$. Then 
$e^{*}_{2}v^{*}e^{*}_{1}$ is a subwalk of $w_{F^{*}_{v}}$ and the {\it different} edges $e^{*}_{1}, e^{*}_{2}$ are consecutive w.r.t. the {\it anti-clockwise} orientation of this facial walk. We say that the $F_{v^{*}}$-sector at $v$, spanned by the pair $(e_{1} ,e_{2})$ and the $F^{*}_{v}$-sector at $v^{*}$ spanned by $(e^{*}_{1} ,e^{*}_{2})$ are in {\it opposite position}; see Fig. \ref{Figure**}. Altogether there are $4r$ of such (ordered) pairs of $\mathcal{G}$-, $\mathcal{G}^{*}$-vertices. Note that if $v$ occurs $p$ times in $w_{F_{v^{*}}}$, then $v^{*}$ shows up also $p$ times in $w_{F^{*}_{v}}$.\\

\begin{figure}[h]
\begin{center}
\includegraphics[scale=0.8]{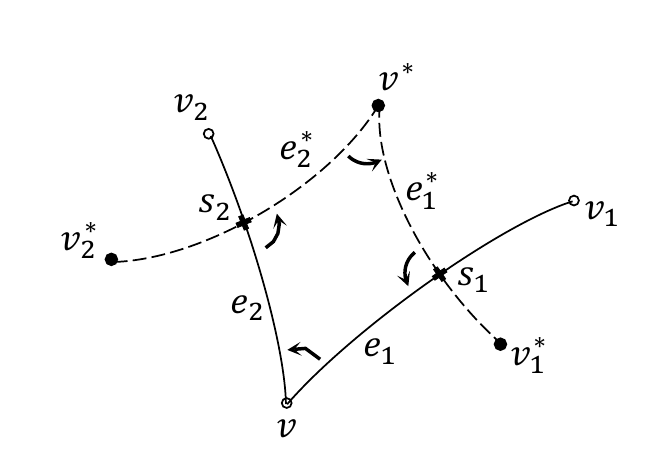}
\caption{\label{Figure**} Pairs of facial sectors in opposite position.}
\end{center}
\end{figure}

\noindent
The next step is to introduce the analogon of the {\it common refinement} $\mathcal{G}(f) \wedge \mathcal{G}^{*}(f)$.\\

\noindent
To this aim:

\begin{definition}
\label{ND7.4}
The abstract graph $\mathbb{P}(\mathcal{G})$ is given as follows:
\begin{itemize}
\item 
There are $4r$ vertices (on three levels) represented by:
\begin{align*}
&\text{- the $\mathcal{G}$-vertices} &\text{[Level-1],}\\
&\text{- the pairs $s=(e,e^{*}), e \in E(\mathcal{G}), e^{*} \in E(\mathcal{G}^{*})$}&\text{[Level-2],}\\
&\text{- the $\mathcal{G}^{*}$-vertices} & \text{[Level-3].}
\end{align*}
\item
There are $8r$ edges:
\begin{align*}
&\text{- a vertex on Level-2, represented by $(e,e^{*})$, 
is connected to two different vertices }\\
&\text{on Level-1, namely the $\mathcal{G}$-vertices incident with $e$, and to two different vertices on}\\
&\text{Level-3, namely the $\mathcal{G}^{*}$-vertices incident with $e^{*}$.}\\
&\text{- vertices on Level-1 are not connected with vertices on Level-3.}
\end{align*}
\end{itemize}
\end{definition}

\noindent
$\mathbb{P}(\mathcal{G})$-vertices on the Levels-1, -3 are denoted as the corresponding $\mathcal{G}-, \mathcal{G}^{*}-$vertices. The graph $\mathbb{P}(\mathcal{G})$ is directed by the convention: vertices on Level-1 (resp. Level-3) are the end- (resp. begin-)points of its edges.

We claim the existence of a cellular embedding of
$\mathbb{P}(\mathcal{G})$ in $T$, denoted $\mathcal{G} \wedge  \mathcal{G}^{*}$, with faces determined by the $4r$ pairs of facial sectors in {\it opposite position}. In order to verify this claim, consider an arbitrary pair of such sectors, given by the subwalks $e_{1}ve_{2}$ and $e^{*}_{1}v^{*}e^{*}_{2}$ with $s_{\ell}=(e_{\ell}, e_{\ell}^{*}), \ell =1,2;$ compare Fig.\ref{Figure**}. We specify local rotation systems on $\mathbb{P}(\mathcal{G})$ at $v$ and $v^{*}$ by $\pi_{v}$ and $\pi^{*}_{v^{*}}$ respectively. The rotation systems at $s_{1}$ and $s_{2}$ are given by the cyclic permutations $(s_{1}v,s_{1}v^{*}_{1}, s_{1}v_{1}, s_{1}v^{*})$, respectively $(s_{2}v,s_{2}v^{*}, s_{2}v_{2}, s_{2}v^{*}_{2})$, where $v_{\ell}$ and $v^{*}_{\ell}$ stand for the vertices incident with $e_{\ell}$ and $e^{*}_{\ell}$ that are different from respectively $v$ and $v^{*}, \ell=1,2.$ The resulting rotation system for $\mathbb{P}(\mathcal{G})$ is called $(\Pi,\Pi^{*})$. Now starting from $vs_{2}$ we find the $(\Pi,\Pi^{*})$-walk
$(vs_{2}, s_{2}v^{*} ,v^{*}s_{1} , s_{1}v)$.

This yields a cellular embedding of $(\mathbb{P}(\mathcal{G}),(\Pi,\Pi^{*}))$ into a surface homeomorphic to $T$ (because the alternating sum of the numbers of vertices, edges and $(\Pi,\Pi^{*})$-walks in $\mathbb{P}(\mathcal{G})$ vanishes). This embedding is denoted by $\mathcal{G} \wedge  \mathcal{G}^{*}$, and can be viewed as the common refinement of $\mathcal{G} $ and $  \mathcal{G}^{*}$. Each face in $\mathcal{G} \wedge  \mathcal{G}^{*}$ is represented by a quadruple of directed edges in $\mathbb{P}(\mathcal{G})$ and is associated with exactly one vertex on Level 1, one vertex on Level 3 (in opposite position) and two vertices on Level 2. Moreover, each $\mathcal{G} $-face ($  \mathcal{G}^{*}$-face) is built up from the sets of all $\mathcal{G} \wedge  \mathcal{G}^{*}$-faces centered at a $  \mathcal{G}^{*}$-vertex ($\mathcal{G}$-vertex), ordered in accordance with the orientation of $\mathcal{G}$ ($\mathcal{G}^{*}$). This observation turns the {\it abstract} graph $\mathbb{P}(\mathcal{G})$ into a distinguished graph $\mathbb{P}^{d}(\mathcal{G})$ with only distinguished sets of {\it Type 1} (in the sense of \cite{Peix2}).

\noindent
Following Peixoto, the distinguished graph $\mathbb{P}^{d}(\mathcal{G})$ is realizable as {\it the} distinguished graph of 
a $C^{1}$-{\it structurally stable} vector field, say\footnote{\label{7FN6}Since $\mathcal{G}^{*}$ is determined by $\mathcal{G}$, we occasionally refer to $\mathcal{G}$ as to the distinguished graph of $ \mathcal{X}(\mathcal{G})$.} $ \mathcal{X}(\mathcal{G})$ on $T$, with:
\begin{itemize}
\item 
as hyperbolic attractors (repellors): the $\mathcal{G}$-vertices ($\mathcal{G}^{*}$-vertices),
\item
as 1-fold saddles: the other $\mathcal{G} \wedge  \mathcal{G}^{*}$-vertices, 
\item
as stable (unstable) separatrices at the saddles: the $\mathcal{G} \wedge  \mathcal{G}^{*}$-edges with as begin point a $\mathcal{G}^{*}$-vertex
(as end point a $\mathcal{G}$-vertex).
\item
as canonical regions (of Type 1): the faces of $\mathcal{G} \wedge  \mathcal{G}^{*}$. 
\end{itemize}
Note that $ \mathcal{X}(\mathcal{G})$ exhibits {\it no} ``saddle connections'', {\it no} closed orbits and thus {\it no} limit cycles. \\

\noindent
In order to specify the roles of $\mathcal{G} $ and $  \mathcal{G}^{*}$, we occasionally write $ \mathcal{X}(\mathcal{G})= \mathcal{X}_{\mathcal{G} \wedge  \mathcal{G}^{*}}$.\\

\noindent
Again, due to Peixoto's classification result (\cite{Peix2}) on structural stability, we have\footnote{\label{7FN7}In fact: $
\mathcal{X}_{\mathcal{G} \wedge  \mathcal{G}^{*}} \sim \mathcal{X}_{\mathcal{H} \wedge  \mathcal{H}^{*}} \Leftrightarrow \mathbb{P}^{d}(\mathcal{G})
\sim \mathbb{P}^{d}(\mathcal{H})
$, where $\sim$ is defined as in (\ref{vgl22}).}:

\noindent
If $\mathcal{H} $ is any connected multigraph such as $\mathcal{G}$
(i.e., cellularly embedded in $T$, the {\it E-
property} holds, all $\Pi$-walks are clockwise oriented, $r$ vertices, 2$r$ edges) then:
$$
\mathcal{X}_{\mathcal{G} \wedge  \mathcal{G}^{*}} \sim \mathcal{X}_{\mathcal{H} \wedge  \mathcal{H}^{*}} \Leftrightarrow \mathcal{G} 
\sim \mathcal{H}, (\text{ and thus } \mathcal{G}^{*}  \sim \mathcal{H}^{*})
$$
where, as in Section \ref{sec6}, in the left hand side $ \sim $ stands for conjugacy and in the right hand side for equivalency: an 
isomorphism between graphs respecting their orientations\footnote{\label{FN24} More precisely: if $\Pi$ and $\Pi^{'}$ are rotation systems for $\mathcal{G}$ resp. $\! \mathcal{H}$, then either $\pi^{'}_{\varphi(v)}=\pi_{v}$ for all $v \in V(\mathcal{G})$, or $\pi^{'}_{\varphi(v)}=\pi_{v}^{-1}$ for all $v \in V(\mathcal{G})$, where $\varphi$ is a homeomorphism on $T$ with $\varphi(\mathcal{G})=\mathcal{H}$. In the first case we call $\varphi$ orientation-preserving and in the second case orientation-reversing.}.\\
\noindent
The flow $\mathcal{X}(\mathcal{G}^{*})$ is the {\it dual version} of $\mathcal{X}(\mathcal{G})$, i.e., $\mathcal{X}(\mathcal{G}^{*})$ is obtained from $\mathcal{X}(\mathcal{G})$ by reversing the orientations of the trajectories of the latter flow, thereby changing repellors into attractors and vice versa. 
Since $(\mathcal{G}^{*})^{*}=\mathcal{G}$,
the dual version of $\mathcal{X}(\mathcal{G}^{*})$ is $\mathcal{X}(\mathcal{G})$. \\

Now, put $\mathcal{H}=-\mathcal{G}^{*}$, 
then:
$$
\mathcal{X}(\mathcal{G}) \sim \mathcal{X}(-\mathcal{G}^{*}) \Leftrightarrow \mathcal{G} \sim -\mathcal{G}^{*}. \text{[{\bf self duality}]}
$$

This observation can be paraphrased as:\\

\noindent
\begin{lemma}
\label{NL7.4} $\mathcal{X}(\mathcal{G})$ is self dual 
iff $\mathcal{G} $ is self dual. 
\end{lemma}

Put $\delta(\mathcal{G})=\{ \delta_{i} ={\rm deg } (v_{i}), v_{i} \in V(\mathcal{G})
\}$ and $ \delta^{*}(\mathcal{G}):=\{ \delta^{*}_{j} ={\rm deg } (v^{*}_{j}), v^{*}_{j} \in V(\mathcal{G}^{*})
\}$, then:

\begin{lemma}
\label{NL7.5}
$
\mathcal{G} \sim - \mathcal{G}^{*} \Leftrightarrow \delta(\mathcal{G})=\delta^{*}(\mathcal{G}) (=\delta(\mathcal{G}^{*})).
$
\end{lemma}
\begin{proof}
Note that the $\delta_{i}$'s, together with the claim ``clockwise'' (``anti-clockwise'') fix the local rotations of $\mathcal{G}$ and $\mathcal{G}^{*}$.
Now the Heffter-Edmonds-Ringel rotation principle together with $(\mathcal{G}^{*})^{*}=\mathcal{G}$ proves the assertion.
\end{proof}
From Lemmata \ref{NL7.4} and \ref{NL7.5} it follows:
\begin{corollary}
\label{NC7.6} There holds:
$
\mathcal{X}(\mathcal{G}) \sim \mathcal{X}(- \mathcal{G}^{*}) \Leftrightarrow \delta(\mathcal{G})=\delta^{*}(\mathcal{G}).
$
\end{corollary}
\vspace{0.5cm}

\noindent
{\large{\bf 3.3 The A(Angle)-property}}\\

\noindent

Recall that $V(\mathcal{G})=\{ v_{1} , \cdots , v_{r} \}$ and $\delta_{i}={\rm deg}(v_{i})$. The $\delta_{i}$ anti-clockwise ordered edges, incident with the vertex
$v_{i}$ are denoted $e_{i(k)},$ $e_{i(\delta_{i}+1)}=e_{i(1)}, k=1, \cdots , \delta_{i}$.
Note that all these edges are different (because $\mathcal{G}$ is a multigraph). Since $T$ is locally homeomorphic to an open disk, it is always possible to re-draw $\mathcal{G}$, thereby respecting $\Pi$ 
such that the Òanti-clockwise measuredÓ angles at $v_{i}$ between $e_{i(k)}$ and $e_{i(k+1)}$, say $2\pi \omega_{i(k)}$, are strictly positive and sum up to $2 \pi$. The resulting graph is again denoted by $\mathcal{G}$. Since $\mathcal{G}$ is a multigraph, we have altogether $4r (=\delta_{1}+ \cdots + \delta_{r})$
``angles'' $\omega_{i(k)}$.
The set of all these angles 
is $A(\mathcal{G})$. The subset of all angles between 
edges 
that are consecutive edges in the $\Pi$-walk $w_{F_{j}}$ that span a $F_{j}$-sector, is called the set of angles 
of $F_{j}$ and will be denoted by $a(F_{j})$.
Finally, for fixed $i$, the set of all  ``angles'' $\omega_{i(k)}, k=1, \hdots , \delta_{i},$ is the ``set $a(v_{i})$ of angles at $v_{i}$''.\\
\noindent
Now, we introduce:
\begin{definition}
\label{ND7.6}
 $\mathcal{G}$ has the {\it A(Angle)-property} if -possibly under a suitable local re-drawing- the angles in $A(\mathcal{G})$ can be chosen such that: 
 \begin{itemize}
\item[$A_{1}:$] 
$\omega_{i(k)}>0$ for all $\omega_{i(k)} \in A(\mathcal{G})$.
\item [$A_{2}:$] 
$\sum_{a(v_{i})}\omega_{i(k)}=1, \text{ for all }i=1, \! \cdots \! , r.$
\item [$A_{3}:$] 
$\sum_{a(F_{j})}\omega_{i(k)}=1, \text{ for all }j=1, \! \cdots \! , r.$
\end{itemize}
\end{definition}
\noindent
Note that Conditions $A_{1}$ and $A_{2}$ can always be fulfilled; the crucial claim is Condition $A_{3}$.\\
\noindent
Moreover, the sets of angles at the vertices $v$ of $\mathcal{G}$ that fulfil the conditions $A_{1}$ and $A_{2}$, fix the anti-clockwise oriented local rotations $\pi_{v}$. Hence, $\mathcal{G}$ is determined by these angles.\\

Let $J$ be an arbitrary {\it non empty} subset of $\{1, \hdots ,r \}$. The subgraph of $\mathcal{G}$ generated by all vertices and edges in the faces $F_{j} , j \in J,$ is denoted by $\mathcal{G}(J)$. An {\it interior vertex} of $\mathcal{G}(J)$ is a vertex of $\mathcal{G}$ that is only incident with $\mathcal{G}$-faces labelled by $J$, whereas a vertex of $\mathcal{G}(J)$ is called {\it exterior} if it is incident with both a face labelled by $J $ and a face {\it not} labelled by $J$. The sets of all interior, respectively all exterior vertices in $\mathcal{G}(J)$ are denoted by Int$\mathcal{G}(J)$ and 
Ext$\mathcal{G}(J)$ respectively. If $J=\{1, \hdots ,r \}$, then $|{\rm Int}\mathcal{G}(J)|=|J|=|V(\mathcal{G}(J))|=|V(\mathcal{G})|(=r)$, where as usual $| . |$ stands for cardinality. \\

\noindent
We have: 
\begin{lemma}
\label{NL7.7}
 Assume that 
 $\mathcal{G}$ fulfils the A-property. Then:
  \begin{equation}
 \label{eq25}
 |{\rm Int} \mathcal{G}(J )|<|J  |<|V(\mathcal{G}(J))|
 , \text{ for all $J$, $\emptyset \neq J \subsetneq \{1, \! \cdots \! ,r \}$}
\end{equation}
\end{lemma}
\begin{proof}
By Definition \ref{ND7.6}
\begin{equation*}
\sum_{j \in J  }\sum_{a(F_{j})}\omega_{i(k)}=|J  |.
\end{equation*}                           
The contribution of any {\it interior} vertex of $\mathcal{G}(J  )$ to the sum in the left-hand side of this equation 
is equal to 1, whereas each {\it exterior} vertex contributes with a number that is strictly between 0 and 1. 
Hence, we are done if- for the subsets $J$ under consideration- we can prove that ${\rm Ext}\mathcal{G}(J  )\neq \emptyset.$ So, assume ${\rm Ext}\mathcal{G}(J  )$ is empty, thus ${\rm Int} \mathcal{G}(J ) \neq \emptyset.$ Let $J^{C}$ be the complement of $J$ in $\{1, \hdots ,r \}$. Thus $\emptyset \neq J^{C} \subsetneq \{1, \hdots ,r \}$ and Ext$\mathcal{G}(J^{C} )$(=Ext$\mathcal{G}(J))=\emptyset$.
Hence, we also have  Int$\mathcal{G}(J^{C} ) \neq \emptyset$.
Now, the connectedness of $\mathcal{G}$ yields a contradiction. 
\end{proof}

\begin{remark}
\label{NR7.8}
If $\mathcal{G}$ has the {\it A-property}, then: 
l.h.s. of (\ref{eq25}) $\Leftrightarrow$ r.h.s. of (\ref{eq25}), so that one of these equalities is redundant. 
\end{remark}

\begin{lemma}
\label{NL7.9}
If $\mathcal{G}$  fulfils $|J  |<|V(\mathcal{G}(J))|$ for all $J$, $\emptyset \neq \!J\! \subsetneq \{1, \! \cdots \! ,r \}$ $($cf. $($\ref{eq25}$)$$)$, then:
\begin{align}
\notag
&\text{Any assignment of an arbitrary vertex $v_{i_{0}}$ to any face $F_{j_{0}}$ adjacent to $v_{i_{0}}$,}\\
\label{eq26}
&\text{can be extended to a bijection $\mathcal{T}: V(\mathcal{G}) \to F(\mathcal{G}),$ with $v \in V(\partial \mathcal{T}(v))$ and }
\\ \notag
&\text{$\mathcal{T}(v_{i_{0}})=F_{j_{0}}$, i.e., the assignment $v_{i_{0}} \mapsto F_{j_{0}}$ can be extended to  a transversal }\\ \notag
&\text{of the vertex sets 
$V(\partial F_{j}), j=1, \cdots ,r.$}
\end{align}
\end{lemma}
\begin{proof}
Consider the vertex set $V(\partial F_{j})$ of $\partial F_{j}$.  Put for $j \in \{1, \! \cdots \! ,r \}, p_{j}=1, \text{ if }j \neq j_{0}, \text{ and }p_{j_{0}}=0$.
For {\it all non empty} subsets $J$ of $\{1, \hdots ,r \}$(i.e. including $J=\{1, \hdots ,r \}$), we have 
\begin{equation*}
|V(\mathcal{G}(J)) \backslash \{v_{i_{0}}\}| \geqslant \sum_{j \in J}p_{j},
\end{equation*}
According to a slight generalization of Hall's theorem on distinct representatives (cf. \cite{Mirsky}), these inequalities 
are necessary and sufficient for the existence of pairwise disjoint sets $X_{1}, \! \cdots \!, X_{r}$, such that
$$
X_{j} \subset V(\partial F_{j})\backslash \{v_{i_{0}}\}   \text{  , with }|X_{j}| = p_{j}, j=1, \! \cdots \! ,r.
$$
Hence, the singletons(!) $X_{j}, j \in \{1, \! \cdots \! ,r \}, j \neq j_{0}$, together with $v_{i_{0}}$ yield the existence of the desired transversal $\mathcal{T}$.
\end{proof}
\noindent
Now, let us re-label 
the angles of 
$\mathcal{G}$ by $x_{\lambda}$, with $\lambda= 1, \! \cdots \! , 4r (=\sum_{i=1}^{r}{\rm deg}(v_{i}))$. We associate with $\mathcal{G}$ a $2r \! \times \!4r$-matrix $M(\mathcal{G})$ with coefficients $m_{\ell \lambda}$:
\begin{equation*}
m_{\ell \lambda}=
\begin{cases}
	1,&	\text{if } \ell= 1, \! \cdots \!, r,   \text{ and }  x_{\lambda} \text{ is an angle at }v_{\ell},\text{ i.e. }x_{\lambda} \text{ in } a(v_{\ell});\\
	1,&	\text{if } \ell= r+1, \! \cdots \! , 2r,  \text{ and }  x_{\lambda} \text{ is an angle in } a(F_{\ell -r});\\
	0, &	\text{ otherwise.}
\end{cases}
\end{equation*}
Apparently, $\mathcal{G}$
has the {\it A-property}
if and only if the following system of $2r$ equations and $4r$ inequalities has a solution:
\begin{equation}
\label{rela1}
\begin{cases}
[M(\mathcal{G})\; \mid
-\!1].(x  \mid 
1)^{T} = (0, ..., 0)^{T}\\
x_{\lambda} > 0,	\lambda=1, \! \cdots \!, 4r
\end{cases}
\end{equation}
Here, $[M(\mathcal{G}) \; \mid
-\!1]$
stands for the matrix $M(\mathcal{G})$ augmented with a $(4r+1)$-st column, each of its elements being equal to $-1$, and 
$(x  \mid 
1) = (x_{1}, ..., x_{4r}, 1)$.\\

\noindent
Basically due to Stiemke's theorem (cf. \cite{Mang}), System (\ref{rela1}) has a solution 
{\it iff} System (\ref{rela2}) below has no solution for which at least one of the inequalities is strict:
\begin{equation}
\label{rela2}\left(
\begin{matrix}
M(\mathcal{G})^{T} \\
----\\
-1 \! \cdots \! -1
\end{matrix}
\right).Z^{T}\geqslant (0, \! \cdots \! ,0)^{T}, \text{ with } Z=(z_{1}, \cdots , z_{i}, \cdots ,z_{r}, \cdots ,z_{r+j}, \cdots ,z_{2r})
\end{equation}
Here,
$$
\left(
\begin{matrix}
M(\mathcal{G})^{T} \\
----\\
-1 \! \cdots \! -\!1
\end{matrix}
\right)
$$
stands for the matrix  $M(\mathcal{G})^{T}$ augmented with a $(4r+1)$-st row, all its coefficients being equal to $-1$.
For $i,j \in \{1, \cdots ,r \}$, the pair $(i,j)$ is called {\it associated}, notation $(i,j) \in {\bf O}$, if $v_{i}$ and $F_{j}$ share an angle.\\

\noindent
Obviously, System (\ref{rela2}) is equivalent with 
\begin{equation}
\label{rela3}
\begin{cases}
z_{i} +z_{r+j}\geqslant 0, \text{ for all }(i,j) \in {\bf O}\\
\sum_{\ell=1}^{2r}z_{\ell}\leqslant 0
\end{cases}
\end{equation}

But now we are in the position to apply Lemma \ref{NL7.9}: 
\begin{lemma}
\label{NL7.10}
Consider a graph $\mathcal{G}$
, not necessarily with the {\it E-property}. 
Then we have
$$
 \mathcal{G} \text{ has the A-property} \Leftrightarrow |J  |<|V(\mathcal{G}(J))| \text{ for all } J, \emptyset \neq J \subsetneq \{1, \! \cdots \! ,r \}.
$$
\end{lemma}
\begin{proof}
``$\Rightarrow$'' See Lemma \ref{NL7.7}.
\newline
``$\Leftarrow$''.
Suppose that $Z=(z_{1}, \! \cdots \! ,z_{2r})$ is a solution of System (\ref{rela3}) for which at least one of the inequalities is not strict.  We lead this assumption to a contradiction. 

Consider an associated pair $(i_{0},j_{0})$. So, the vertex $v_{i_{0}}$ and the face $F_{j_{0}}$ have an angle in common. Extend by Lemma \ref{NL7.9}, the assignment $v_{i_{0}} \mapsto F_{j_{0}}$ to a transversal $\mathcal{T}$ as in (\ref{eq26}) and define $\tau (i)$ by $F_{\tau (i)}=\mathcal{T}(v_{i})$. This means that $v_{i}$ and $F_{\tau (i)}$ share an angle, thus $(i,\tau(i)) \in {\bf O}$; in particular $(i_{0},\tau(i_{0}))=(i_{0},j_{0}) \in {\bf O}$.
Since $Z$ fulfills (\ref{rela3}), we have:
$z_{i}+z_{r+\tau (i)} \geqslant 0, i=1, \cdots ,r,$ and moreover (use that $\mathcal{T}$ is bijective) also 
$$
\sum_{i=1}^{r}(z_{i}+z_{r+\tau (i)})= \sum_{\ell =1, \hdots ,2r}z_{\ell} \leqslant 0.
$$
Hence, $z_{i}+z_{r+\tau (i)}= 0, i=1, \cdots ,r.$ In particular,
$
z_{i_{0}}+z_{r+j_{0}}=0.
$
Since the associated pair $(i_{0},  j_{0})$ was  chosen arbitrarily, we have $z_{i} +z_{r+j}=0$, for every combination $(i, j) \in {\bf O}$. This contradicts our assumption on $Z$.
It follows that System (\ref{rela3}) {\it does not have} a solution for which at least one of the inequalities is strict. Thus System (\ref{rela1})
{\it does admit} a solution, i.e. ($\mathcal{G}$, $\Pi$) has the {\it A-property}.
\end{proof}

\begin{corollary}
\label{NC7.11}
Let $\mathcal{G}$ be a graph
as in Lemma \ref{NL7.10}.
Then there holds:
$$
\text{$\mathcal{G}$
has the A-property} \Leftrightarrow \text{ Condition $($}\ref{eq26} \text{$)$ holds for $\mathcal{G}$.}
$$
\end{corollary}
\begin{proof}
$\Rightarrow$ See Lemmas \ref{NL7.7}, \ref{NL7.9}.\\
$\Leftarrow$ Follows from  the ($\Leftarrow$ part) of the proof of Lemma \ref{NL7.10}.
\end{proof}
The (equivalent) conditions ``$|J  |<|V(\mathcal{G}(J))|$ for all $J$, $\emptyset \neq J \subsetneq \{1, \! \cdots \! ,r \}$'' and (\ref{eq26}) will be referred to as to the {\it H(Hall)-condition}; see also Section \ref{ss13.2}.

As it is  easily verified, the graphs $\mathcal{G}$ in Fig. \ref{Figure15}(i), (ii) fulfil the {\it H-condition}, but $\mathcal{G}$ in Fig. \ref{Figure15}(iii) not. Hence, by Lemma \ref{NL7.10}, or Corollary \ref{NC7.11}, the graphs $\mathcal{G}$ in Fig. \ref{Figure15}(i), (ii) have the {\it A-property}, but this it not so for the graph in Fig. \ref{Figure15}(iii).\\

\noindent
{\large{\bf 3.4 Newton graphs}}\\

\noindent
\begin{definition}
\label{ND7.12}
Cellularly embedded toroidal graphs with $r$ vertices, 2$r$ edges (and thus $r$ faces) that fulfil the {\it A-and the E-properties} are called {\it Newton graphs of rank $r$}.
\end{definition}
\begin{lemma}
\label{NL7.13}
If ${\rm (}\mathcal{G}$, $\Pi {\rm )}$
is a Newton graph, then this is also true for ${\rm (}\mathcal{G}^{*}$, $\Pi^{*} {\rm )}$.
\end{lemma}
\begin{proof}
In view of Lemma \ref{NL7.3}, we only have to show that ($\mathcal{G}^{*}$, $\Pi^{*}$) has the {\it A-property}. 
Let $v^{*}_{0}$ be a $\mathcal{G}^{*}$-vertex and consider an assignment $v^{*}_{0} \mapsto F^{*}_{v_{0}}$, where $F^{*}_{v_{0}}$ is a
$\mathcal{G}^{*}$-face adjacent to $v^{*}_{0}$ corresponding with the $\mathcal{G}$-vertex $v_{0}$. 
So the pair $(v_{0}, v_{0}^{*})$ is in opposite position, and $v_{0}$ is adjacent to the $\mathcal{G}^{*}$-face $F_{v_{0}^{*}}$.
By assumption, $\mathcal{G}$ fulfills the A-property. So, we can extend (by Corollary \ref{NC7.11}) the 
assignment 
$v_{0} \mapsto F_{v^{*}_{0}}$ to a transversal of the vertex sets of $\mathcal{G}$ (i.e., to pairs ($v_{i}$, $v_{i}^{*}$) in opposite position such that all $v_{i}$ and $v_{i}^{*}$ are different), and thus to a transversal $v_{i}^{*} \to F_{v_{i}}^{*}$ of the vertex sets of  $\mathcal{G}^{*}$-faces (extending $v^{*}_{0} \mapsto F_{v^{*}_{0}}$).
Now, application of Corollary \ref{NC7.11} yields the assertion. 
\end{proof}
\noindent
The above result is easily verified by a geometric argument. Consider -under the assumption that the {\it A-and E-properties} hold for $\mathcal{G}$- the graph $\mathcal{G} \wedge \mathcal{G}^{*}$ on $T$ and proceed in two steps: (see Fig. \ref{Figure18}
)\\
\noindent
Step 1: Re-draw $\mathcal{G} \wedge \mathcal{G}^{*}$ locally around the vertices of $\mathcal{G}$ (solid lines) such that the angles in $A(\mathcal{G})$ fulfil the Conditions A1-A3 (in Definition \ref{ND7.6}).\\
\noindent
Step 2: Due to Condition $A_3$  for $\mathcal{G}$, we may re-draw $\mathcal{G} \wedge \mathcal{G}^{*}$ locally around the vertices of $ \mathcal{G}^{*}$ (dotted lines) such that the $A(\mathcal{G})$- and $A(\mathcal{G}^{*})$-angles of facial sectors in opposite position are equal.

We conclude that also $\mathcal{G}^{*}$ has the {\it A-property}, and find as a by-product:
\begin{lemma}
\label{NL7.14}
If 
$\mathcal{G}$ is a Newton graph, we may assume -possibly after a suitable local re-drawing- that in each face of $\mathcal{G} \wedge \mathcal{G}^{*}$ the angles at the $\mathcal{G}$-and $\mathcal{G}^{*}$-vertices are equal (and non-vanishing).
\end{lemma}

\begin{figure}[h]
\begin{center}
\includegraphics[scale=0.6]{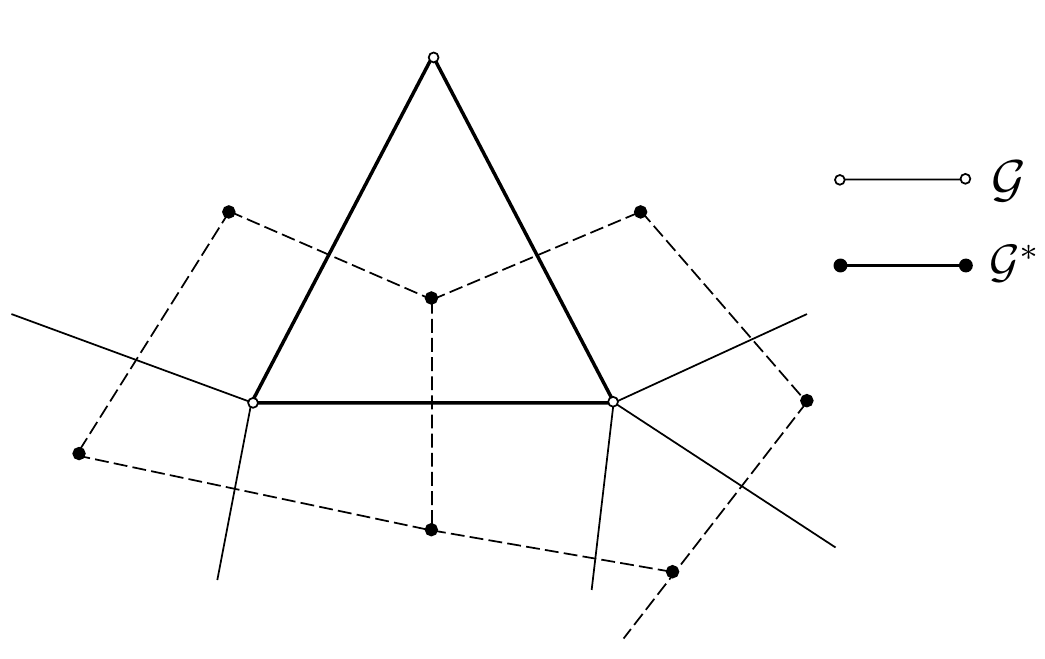}
\caption{\label{Figure18} ($\mathcal{G}$, $\Pi$) and its dual ($\mathcal{G}^{*}$, $\Pi^{*}$); partial} 
\end{center}
\end{figure}

\noindent
From now on we assume that a Newton graph and its dual are always oriented as $\mathcal{G}$ and $\mathcal{G}^{*}$ in Subsection 3.1.\\
\noindent
From Corollary \ref{C6.6N} and Lemma  \ref{L6.7N}
it follows:
\begin{corollary}
\label{NC7.15}
$\mathcal{G}(f) \text{ and } \mathcal{G}^{*}(f)
, f \in \tilde{E}_{r}$, are Newton graphs.
\end{corollary}
In the forthcoming section
we prove that in a certain sense the reverse is also true.\\

\noindent
We end up this section with a lemma that we will use in the sequel:

\begin{lemma}
\label{NL7.16}
 Let $\mathcal{G}$
 be of order $r=2$ or 3. Then:
If $r=2$, the A-property always holds, whereas in Case $r=3$ the {\it E-property} implies the A-property.  
\end{lemma}
\begin{proof}
Let $J$ be an arbitrary {\it non empty, proper} subset of $\{ 1, \cdots ,r \}$. \\
\underline{Case $r=2$}: Note that $|J|=1$, thus 
$|V(\mathcal{G}(J)|>1$
(because $\mathcal{G}$ has no loops). So we have: $|V(\mathcal{G}(J)|>|J|$, i.e., the {\it H-condition} holds, and Lemma \ref{NL7.10} yields the assertion.\\
\underline{Case $r=3$}:\\
If $|J|=1$, then $|V(\mathcal{G}(J)|>1$(because $\mathcal{G}$ has no loops), thus $|V(\mathcal{G}(J)|>|J|$.\\
If $|J|=2$, then $|J^{c}|=1$ and $|V(\mathcal{G}(J^{c})|  \geqslant 2$ (since $\mathcal{G}$ has no loops). Moreover, by the {\it E-property}, each edge must be adjacent to at least two faces. This implies: Int$\mathcal{G}(J^{c})=\emptyset$. 
Thus $|{\rm Ext}\mathcal{G}(J)|=|{\rm Ext}\mathcal{G}(J^{c})|=|V(\mathcal{G}(J^{c})| \geqslant 2$  and $|V(\mathcal{G}(J)|=|{\rm Ext}\mathcal{G}(J)|+|{\rm Int}\mathcal{G}(J)|.$
Distinguish now between two cases: 
\begin{itemize}
\item ${\rm Int}\mathcal{G}(J) \neq \emptyset,$ then $|V(\mathcal{G}(J)|>|J|$.
\item ${\rm Int}\mathcal{G}(J) = \emptyset,$ then the {\it three} vertices of $\mathcal{G}$ must be exterior vertices for $\mathcal{G}(J)$, thus also $|V(\mathcal{G}(J)|>|J|$. Hence, $|V(\mathcal{G}(J)|>|J|$ holds for all $J$ under consideration, and Lemma \ref{NL7.10} yields the assertion.
\end{itemize}

\end{proof}

\begin{remark}
\label{NR7.17}
In contradistinction to the {\it A-property}, the {\it E-property} does not hold for all second order multigraphs $\mathcal{G}$
; compare Fig. \ref{Figure19} (i), where the dual $\mathcal{G}^{*}$
admits a loop. 
From Fig. \ref{Figure19} (ii), it follows that Lemma \ref{NL7.16} ($r=3$)
is not true in the case that $r=4$.
From Fig. \ref{Figure15} (ii) we learn that the {\it A-property} does not imply the {\it E-property}, even if $r=3$.
\end{remark}

\begin{figure}[h!]
\begin{center}
\includegraphics[scale=0.6]{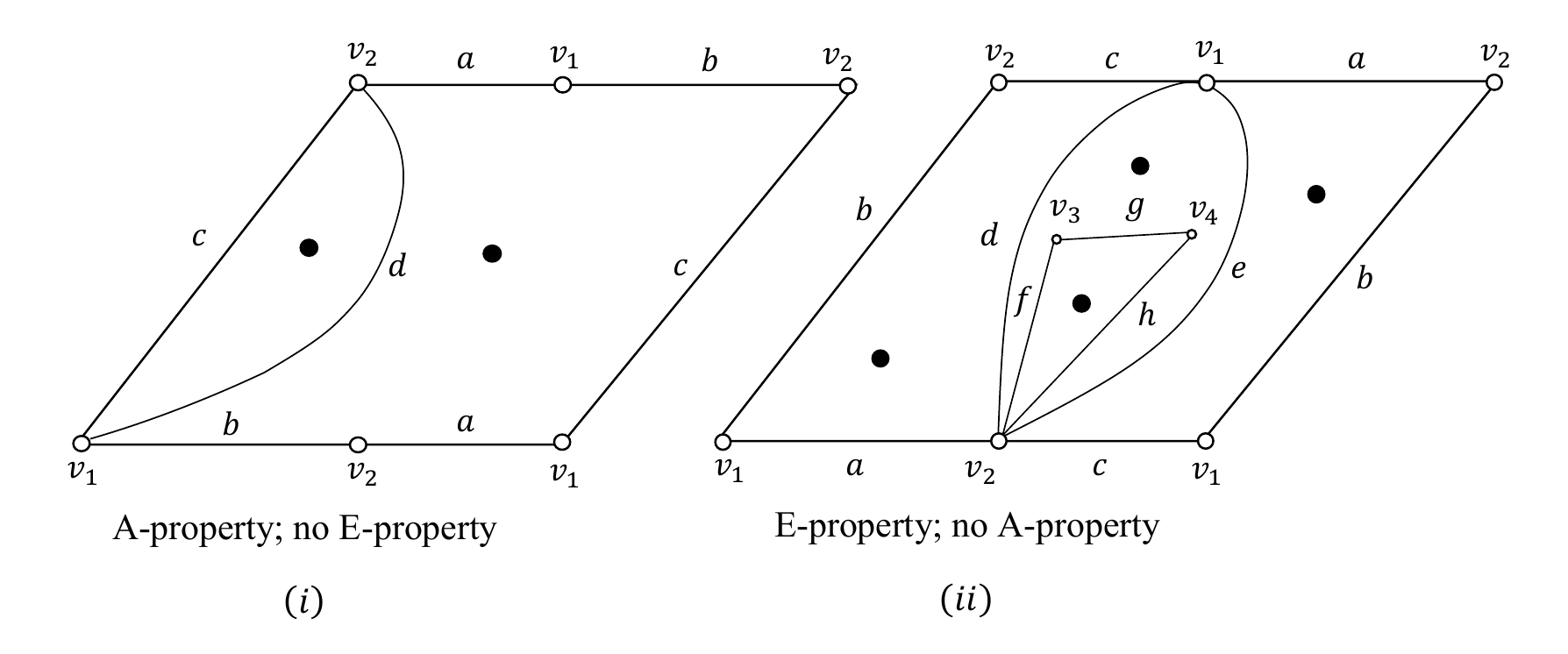}
\caption{\label{Figure19} Two graphs $\mathcal{G}$.}
\end{center}
\end{figure}

\newpage
\section{Structurally\,stable\,elliptic\,Newton\,flows:\,Representation}
\label{Nsec8}

In this section we prove that the reverse of Corollary \ref{NC7.15}  is also true. 
\begin{theorem}
\label{NT8.1}
Any Newton graph $\mathcal{G}$ of order $r$ 
can be realized - up to equivalency - as the graph $\mathcal{G}(f), f \in \tilde{E}_{r}$. 
\end{theorem}
\begin{proof}
Based on several steps, see the end of this section.
\end{proof}
Starting point is an arbitrary Newton graph $\mathcal{G}$. We apply the results of Section \ref{Nsec7}. Let $\mathcal{X}(\mathcal{G})$ be a $C^{1}$-structurally stable vector field on $T$ without limit cycles, determined - up to conjugacy
- by 
$\mathcal{G} \wedge  \mathcal{G}^{*}$,
thus by $\mathcal{G}$ as the ``distinguished graph'' of $\mathcal{X}(\mathcal{G})$.
(cf. Footnote \ref{7FN6})

  The flow $\mathcal{X}(\mathcal{G})$ is gradient like, i.e. up to conjugacy equal
to a gradient flow (with respect to a $C^{1}$-Riemannian
metric $R$ on $T$). This can be seen as follows:\\
  An arbitrary equilibrium, say {\bf x}, of the (structurally stable!) flow $\mathcal{X}(\mathcal{G})$ is of hyperbolic type, i.e. the derivative $D_{\text{{\bf x}}}\mathcal{X}(\mathcal{G})$ has eigenvalues $\lambda_{1}(\text{{\bf x}}), \lambda_{2}(\text{{\bf x}})$ with non-vanishing real parts, cf. \cite{Peix1}.
By the Theorem of Grobman-Hartman (cf. \cite{Hart}) we have: (use also Theorem 8.1.8, Remark 8.1.10 in \cite{JJT1}): On a suitable 
 {\bf y}-coordinate neighborhood [with {\bf y}=$(y_{1}, y_{2})^{T}$] of {\bf x}, the phase portrait of $\mathcal{X}(\mathcal{G})$ is conjugate
 with the phase portrait around  {\bf y}= {\bf 0}
 of one of the flows given by: 
$$
\text{{\bf y}}'=- \left(
\begin{matrix}
\lambda_{1}&0 \\
0&\lambda_{2}
\end{matrix}
\right)\text{{\bf y}},\; \text{{\bf y}}(\text{{\bf 0}})=\text{{\bf 0}}, \text{ with either }\lambda_{1}=\lambda_{2}=1, \text{ or } \lambda_{1}=\lambda_{2}=-1 \text{ or }\lambda_{1}=-\lambda_{2}=1,
$$
corresponding to the cases where {\bf y}= {\bf 0}
stands for respectively a {\it stable star node}, an {\it unstable star node} and an {\it orthogonal saddle}; see Comment on Fig.\ref{Figure1} and Fig.\ref{Figure8.1}.

\begin{figure}[h]
\begin{center}
\includegraphics[scale=0.76]{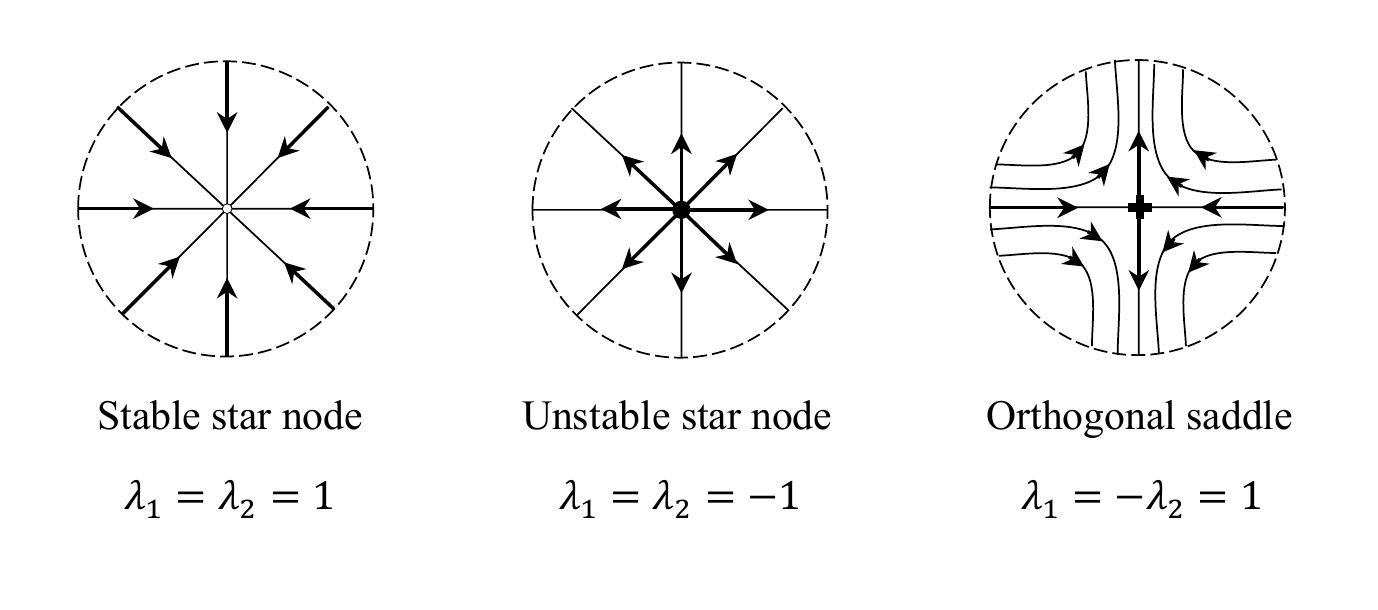}
\caption{\label{Figure8.1}The possible local phase portraits of $\mathcal{X}(\mathcal{G})$ around {\bf y}= {\bf 0}.}
\end{center}
\end{figure}

Applying a flow box argument (``cutting'' and ``pasting'' of local phase portraits), we find that $\mathcal{X}(\mathcal{G})$ is conjugate
  with a structurally stable  smooth flow on $T$  Ð again denoted by $\mathcal{X}(\mathcal{G})$ Ð  with as equilibria: 2$r$ star nodes ($r$ of them being stable, the other $r$ unstable) and 2$r$ orthogonal saddles.
  The underlying ``distinguished graph'' is denoted - again - by $\mathcal{G}$. It follows that the angle between two $\mathcal{G}$-edges (i.e. unstable separatrices at saddles for $\mathcal{X}(\mathcal{G})$ that are incident with the same $\mathcal{G}$-vertex (i.e. a stable star node for $\mathcal{X}(\mathcal{G})$), may assumed to be well-defined and nonvanishing. 
   
We adopt the notations\! \!/ \!\!conventions as introduced in the preambule to Definition \ref{ND7.6} ({\it Angle Property}).
In particular, let the $\mathcal{G}$-vertex $v_{i}$ stand for a stable node of $\mathcal{X}(\mathcal{G})$.
 In Fig. \ref{Figure8.2}-(a) we present a picture of $\mathcal{X}(\mathcal{G})$ w.r.t. the {\bf y}-coordinates around {\bf 0} ($=v_{i}$).
 Here the bold lines stand for $\mathcal{G}$-edges, and the thin lines for other $\mathcal{X}(\mathcal{G})$-trajectories on a small disk $D$ around {\bf y}= {\bf 0}. Note that the ÒanglesÓ $\omega_{i(k)}$ in this figure fulfil the conditions $A_{1}, A_{2}$  in Definition \ref{ND7.6}. 
In Fig. \ref{Figure8.2}-(b), we consider a similar configuration of $\mathcal{X}(\mathcal{G})$-trajectories on $D$, approaching $v_{i}$, with as only additional condition that the tuples $(e_{i(1)}, \cdots , e_{i(\delta_{i})})$ and $(e_{i(1)}', \cdots , e_{i(\delta_{i})}')$
are equally ordered.   
Consider the oriented arcs  arc($i(k), i(k+1)$) and $\text{arc}^{'}(i(k), i(k+1)$) in the boundary $\partial D$ of  $D$, determined by respectively the consecutive pairs $(e_{i(k)}, e_{i(k+1)})$ and $(e_{i(k)}', e_{i(k+1)}')$. 
Under suitable shrinking/stretching, these arcs can be identified. This yields an orientation preserving homeomorphism $\psi$ from $\partial D$ onto itself. It is easily proved that $\psi$ can be extended to a homeomorphism $\Psi : D \to D$
mapping the  $\mathcal{X}(\mathcal{G})$-trajectories in Fig. \ref{Figure8.2}-(a) onto those in Fig. \ref{Figure8.2}-(b). This procedure will be referred to as a local re-drawing (around $v_{i}$).

With the aid of local re-drawings, together with a ``cut'' and ``paste'' constructionÓ, the pair ($\mathcal{X}(\mathcal{G}), \mathcal{G}$)  can be changed into an equivalent structurally stable flow (again denoted $\mathcal{X}(\mathcal{G})$) and an equivalent ÒdistinguishedÓ graph (again denoted $\mathcal{G}$), with pictures as Fig. \ref{Figure8.2}-(b) instead of Fig. \ref{Figure8.2}-(a). We conclude that the angles $\omega_{i}(k)$ in Fig. \ref{Figure8.2}-(a) may be altered arbitrarily (provided that the Conditions $A_{1}, A_{2}$  in Definition \ref{ND7.6} persist) without changing the topological types of $\mathcal{X}(\mathcal{G})$ and $\mathcal{G}$.

\begin{figure}[h]
\begin{center}
\includegraphics[scale=0.8]{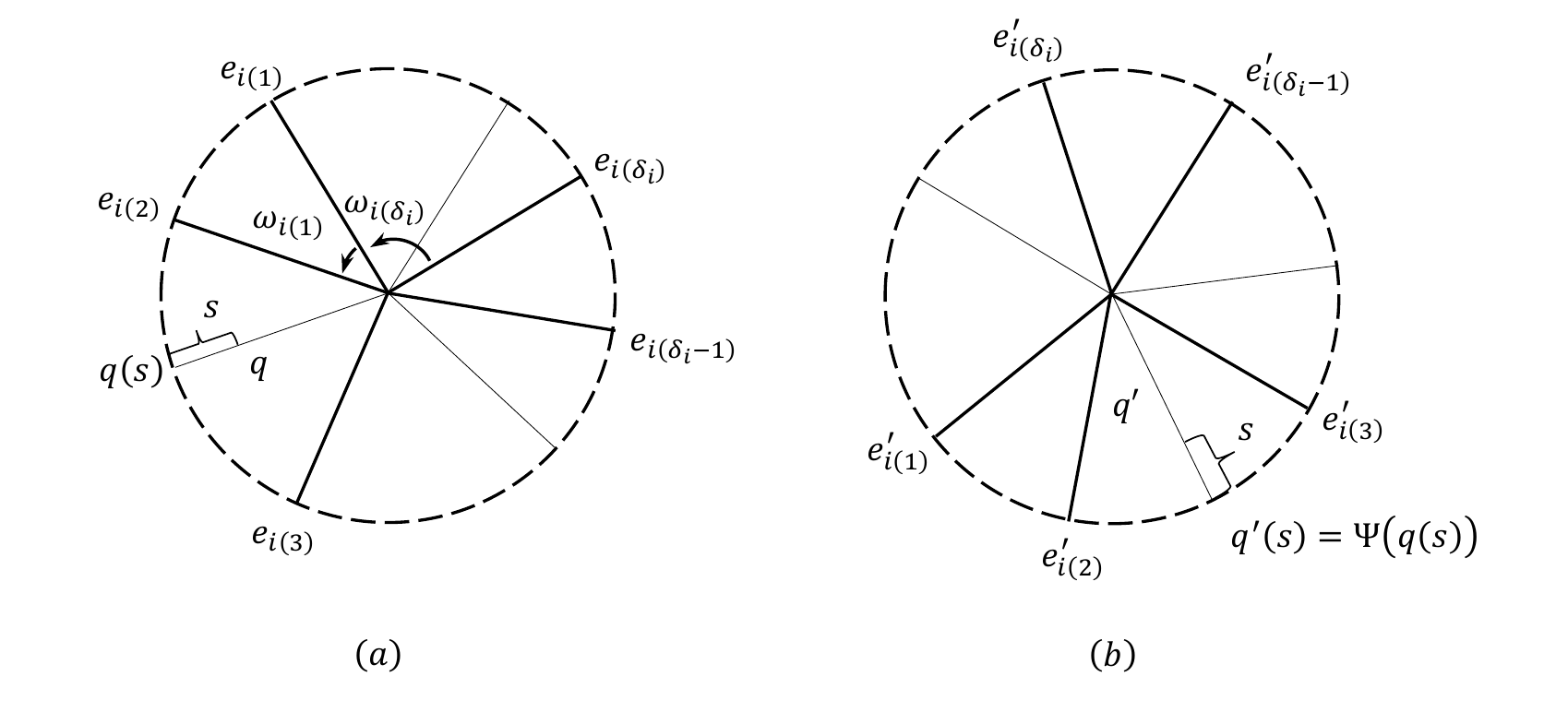}
\caption{\label{Figure8.2}Local phase portraits of $\mathcal{X}(\mathcal{G})$ around a stable star node before and after local redrawing.}
\end{center}
\end{figure}

Note that any toroidal graph, equivalent 
with a Newton graph (such as the original graph $\mathcal{G}$), is also a Newton graph (cf. Defintion \ref{ND7.2} and Lemma \ref{NL7.10}). Moreover, not only $\mathcal{G}$, but also $\mathcal{G}^{*}$ is Newtonian (cf. Lemma \ref{NL7.13}). Hence, compare (the proof of ) Lemma \ref{NL7.14}, with the aid of local re-drawings around the vertices of $\mathcal{G}$ and $\mathcal{G}^{*}$, together with a ``cut and past construction'', it is easily shown that: 
\begin{align*}
&\text{In each face of $\mathcal{G} \wedge \mathcal{G}^{*}$ (= canonical $\mathcal{X}(\mathcal{G})$-region), the angles at the $\mathcal{G}$- and $\mathcal{G}^{*}$-vertex}\\
&\text{(= a stable, respectively unstable, star node of $\mathcal{X}(\mathcal{G})$) are equal and non-vanishing.}
\end{align*}  
 With respect to the various local $\text{{\bf y}}$-coordinate systems around the $\mathcal{X}(\mathcal{G})$-equilibria 
, we define the $4r$ functions $h_{i}, h_{i}^{*},h_{j}^{**}, i=1, \cdots ,r, j=1, \cdots , 2r$, as follows:
\begin{align*}
&h_{i}(\text{{\bf y}})=\frac{1}{2}(y_{1}^{2}+y_{2}^{2}), \text{ in case of stable nodes at $\text{{\bf y}}=\text{{\bf 0}}$ representing the $r$ vertices $v_{i}$ of $\mathcal{G}$,}\\
&h_{i}^{*}(\text{{\bf y}})=-\frac{1}{2}(y_{1}^{2}+y_{2}^{2}),\text{in case of unstable nodes at $\text{{\bf y}}=\text{{\bf 0}}$ representing the $r$ vertices $v_{i}^{*}$ of $\mathcal{G}^{*}$,}\\
&h_{j}^{**}(\text{{\bf y}})=\frac{1}{2}(y_{1}^{2}-y_{2}^{2}) \text{ in case of saddles at $\text{{\bf y}}=\text{{\bf 0}}$ representing the $2r$ edges $e_{j}$ of $\mathcal{G}$.}
\end{align*}

Note that each function exhibits a non-degenerate critical point at $\text{{\bf y}}=\text{{\bf 0}}$. Moreover, on a $\text{{\bf y}}$-coordinate neighborhood $N$ around an equilibrium of $\mathcal{X}(\mathcal{G})$, the vector field $\mathcal{X}(\mathcal{G})$ is the negative gradient vector field [w.r.t. the standard Riemannian structure on $N$] of the associated function. 
  Apparently, the flow $\mathcal{X}(\mathcal{G})$, being structurally stable on $T$ (without limit cycles), together with the functions $h_{i}, h_{i}^{*}$  and $h_{j}^{**}$, fulfils the Requirements (1)-(4) laid upon Theorem B in \cite{Smale}. So, applying this theorem we may conclude that there is a 
 function $h$ on $T$ such that:
  \begin{enumerate}
\item[1.] 
The critical points of $h$ coincide with the equilibria of $\mathcal{X}(\mathcal{G})$ and $h$ coincides with the functions $h_{i}, h_{i}^{*}$, $h_{j}^{**}$ plus a constant in some neighborhood of each critical point.
\item[2.] $Dh(x)\cdot \mathcal{X}(\mathcal{G})|_{x} <0$ outside the critical point set Crit($h$) of $h$.
\item[3.] The function $h$ is {\it self indexing}, i.e., the value of $h$ in a critical point $\beta$ equals the Morse index of $\beta$ (= $\#$(negative eigenvalues of $D^{2}h (\beta )$). Thus: $h(\beta)= 0 (=2)$, in case of a stable (unstable) node and  $h(\beta)= 1$ in case of a saddle.
\end{enumerate}
As a corollary, we have (cf. Theorem 8.2.8 in \cite{JJT1}, and \cite{Smale}),
there is a variable (Riemannian) metric $R(\cdot )$ on $T$, such that:
$$
{\rm grad}_{R} h =\mathcal{X}(\mathcal{G}),
$$
where ${\rm grad}_{R} h$ is a vector field on $T$ of the form: (w.r.t. local coordinates {\bf x}
for $T$)
$$
{\rm grad}_{R} h(\text{{\bf x}})=-R^{-1}(\text{{\bf x}}) D^{T}h(\text{{\bf x}}).
$$
Here $R(\text{{\bf x}})$ is a symmetric, positive definite $2 \times 2$-matrix, with coefficients 
depending in a $C^{1}$-fashion on $\text{{\bf x}}$.
   Note that the direction of ${\rm grad}_{R} h $ is uniquely determined by the above transversality Condition 2., whereas on the neighborhoods $N$ around the $\mathcal{X}(\mathcal{G})$-equilibria, the matrices $R(\cdot )$ are just the $2 \times 2$-unit matrix $I_{2}$. Moreover, ${\rm grad}_{R} h(\text{{\bf x}}) \neq 0$,
if and only if $\text{{\bf x}}$ is outside the set Crit($h$) ($=$ set of $\mathcal{X}(\mathcal{G})$-equilibria). 
  
For $\text{{\bf x}} \notin $ Crit($h$), let ${\rm  grad}_{R(\text{{\bf x}})}^{\perp}h(\text{{\bf x}}) \neq 0$, be a vector $R$-orthogonal to ${\rm  grad}_{R(\text{{\bf x}})}h(\text{{\bf x}})$, i.e.
\begin{equation}
\label{Vgl30}
({\rm  grad}_{R(\text{{\bf x}})}^{\perp}h(\text{{\bf x}}))^{T}.R(\text{{\bf x}}).(-R^{-1}(\text{{\bf x}}). D^{T}h(\text{{\bf x}}))[=-Dh(\text{{\bf x}}). {\rm  grad}_{R(\text{{\bf x}})}^{\perp}h(\text{{\bf x}})]=0.
\end{equation}

Let $x_{0}$  be a point in the level set $L_{c}=\{ x \in T \mid h(x)=c; c=\text{ constant}\}$. Then we have
\begin{itemize}
\item Assume $x_{0} \notin$ Crit($h$), thus $L_{c}$ is, locally around $x_{0}$, a regular curve. By (\ref{Vgl30}) this local curve is $R$-orthogonal to the trajectory of $\mathcal{X}(\mathcal{G})$ ($={\rm grad}_{R}h$) through $x_{0}$. 
\item If $x_{0}\in$ Crit($h$), then $x_{0}$ is either an isolated point, ÒsurroundedÓ by closed regular curves $L_{c}, c \neq 0,2$ (in the case where $x_{0}$ is a $\mathcal{X}(\mathcal{G})$-node), or a Òramification pointÓ at the intersection of two (orthogonal) components of $L_{1}$ (in case of a $\mathcal{X}(\mathcal{G})$-saddle). This follows from the fact that on the neighborhoods $N$ around the equilibria of $\mathcal{X}(\mathcal{G})$, the Riemannian metric $R$ is just the standard one. 
\end{itemize}
 So, we may subdivide the level sets $L_{c}$ into the disjunct union of ÒmaximalÓ regular curves (to be referred to as to the Òlevel lines $L_{c}$Ò) and single points (in Crit($h$)). 
 Let $\text{{\bf x}}(t), \text{{\bf x}}(0) \notin $ Crit($h$) be a trajectory for $\mathcal{X}(\mathcal{G})$ ($={\rm grad}_{R}h$). Since $R^{-1}(\text{{\bf x}})$ is a symmetric, positive definite matrix:
 \begin{equation}
\label{Vgl31}
\frac{d}{dt}h(\text{{\bf x}}(t))|_{t=0}=Dh(\text{{\bf x}}(0)).\text{{\bf x}}'(0)=Dh(\text{{\bf x}}(0)).(-R^{-1}(\text{{\bf x}}(0))).D^{T}h(\text{{\bf x}}(0))) <0
\end{equation}
So, $h(\text{{\bf x}}(t))$ decreases when $t$ increases, and by the indexing Condition 3: $0 \leqslant h(x) \leqslant 2,$
for all $x \in T$.  
By (\ref{Vgl31}), when travelling along  the boundary of an open canonical $\mathcal{X}(\mathcal{G})$-region[=$\mathcal{G} \wedge  \mathcal{G}^{*}$-face], say $\overline{\overline{\mathcal{R}} }_{ij}$ in Fig. \ref{Figure8.3}, the functional values of $h$ vary strictly from 2 (at the unstable node $v_{j}^{*}$) via 1 (at a saddle $\sigma_{1}$ or $\sigma_{2}$) to 0 (at the stable node $v_{i}$). From this it follows -use also the transversality Condition 2 - that a level line $L_{c}$ , entering $\overline{\overline{\mathcal{R}} }_{ij}$  through the boundary $\partial \overline{\overline{\mathcal{R}} }_{ij}$  between $v_{i}$ and $\sigma_{1}$  [thus $0< c <1$], must leave this region through $\partial \overline{\overline{\mathcal{R}} }_{ij}$  between $v_{i}$ and $\sigma_{2}$ . Also: if $L_{c}$  enters $\overline{\overline{\mathcal{R}} }_{ij}$ through $\partial \overline{\overline{\mathcal{R}} }_{ij}$ between $v_{j}^{*}$ and $\sigma_{1}$ [thus $1< c <2$], then it leaves  $\overline{\overline{\mathcal{R}} }_{ij}$ through $\partial \overline{\overline{\mathcal{R}} }_{ij}$ between $v_{j}^{*}$ and $\sigma_{2}$. By the same argumentation: the saddles $\sigma_{1}$ and $\sigma_{2}$  are connected by a level line $L_{1}$. Considering unions of all $\mathcal{G} \wedge  \mathcal{G}^{*}$-faces incident with the same vertex representing a stable (unstable) attractor of $\mathcal{X}(\mathcal{G})$, we find: the level sets $L_{c}, c \neq 0, 1 \text{ or } 2,$ are closed smooth regular curves, either contractable to a stable attractor [in case $0<c<1$], or  to an unstable attractor [in case $1<c<2$] . 
  Altogether, a level line $L_{c}$ is either a closed curve, or it connects two different $\mathcal{X}(\mathcal{G})$-saddles.  Hence, the following definition makes sense: (compare also Definition \ref{D6.17N})

\begin{figure}[h]
\begin{center}
\includegraphics[scale=0.65]{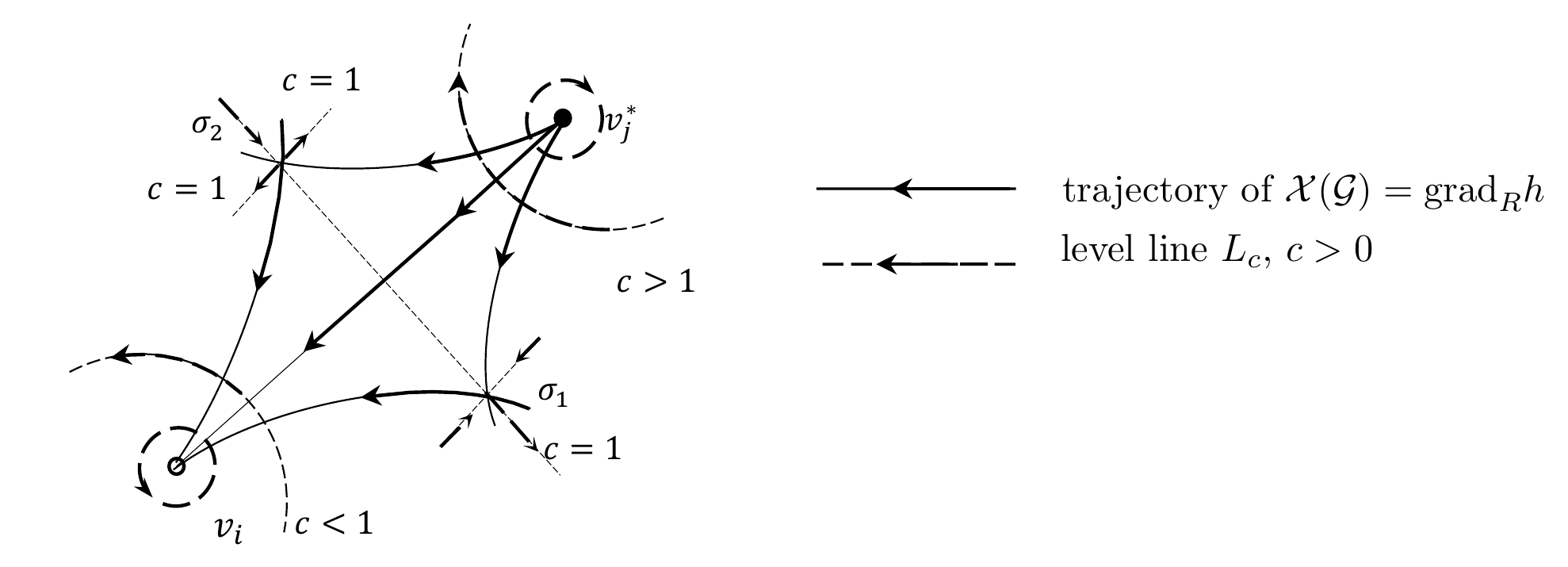}
\caption{\label{Figure8.3}The canonical $\mathcal{X}(\mathcal{G})$-region $\overline{\overline{\mathcal{R}} }_{ij}$. }
\end{center}
\end{figure}

 \begin{definition}
\label{ND8.2} The graph $\mathcal{G}^{\perp}$ on the torus $T$ is given by:
\begin{itemize}
\item Vertices are the $2r$ saddles for $\mathcal{X}(\mathcal{G})$ on $T$.
\item Edges are the $4r$ level lines $L_{1}$ connecting different saddles of $\mathcal{X}(\mathcal{G})$.
\end{itemize}
\end{definition}
 Apparently, $\mathcal{G}^{\perp}$ is cellularly embedded, and Ð by Euler's formula Ð  this graph is connected (since there are $2r$ faces, determined by the stable and unstable nodes of  $\mathcal{X}(\mathcal{G})$). So, the function $h$ admits the same value on (the edges and vertices of) $\mathcal{G}^{\perp}$, whereas Ð by the self indexing Condition 3 - we know that this value equals 1. Thus, the embedded graph $\mathcal{G}^{\perp}$ Ð as a point set in $T$ Ð is just the level set $L_{1}$. This leads to Fig. \ref{Figure8.4}, where we present the graphs $\mathcal{G}$, $\mathcal{G}^{*}$ and $\mathcal{G}^{\perp}$, together with some more level lines $L_{c}$. We endow the level lines $L_{c}, 0<c<1$ (the level lines $L_{c}, 1<c<2$ ) with the anti-clockwise (clock-wise) orientation. Doing so, we can turn $\mathcal{G}^{\perp}$ into a oriented graph; see Fig. \ref{Figure8.4}. 
 
We fix the vector field ${\rm  grad}_{R(\text{{\bf x}})}^{\perp}h(\text{{\bf x}})$ by demanding that it has the same length as ${\rm grad}_{R} h(\text{{\bf x}})$ (w.r.t. the norm, induced by $R(\text{{\bf x}})$) and is oriented according to the orientation of the level line $L_{c}$ through $\text{{\bf x}}$, see Fig. \ref{Figure8.3}. So, by (\ref{Vgl30}), we may interpret the net of $\mathcal{X}(\mathcal{G})$-trajectories and level lines $L_{c}$, as the $R$-orthogonal net of trajectories for the vector fields ${\rm grad}_{R}h$ and ${\rm  grad}_{R}^{\perp}h$. The switch from $\mathcal{X}(\mathcal{G})$ to $\mathcal{X}(\mathcal{G}^{*})$ (=$-\mathcal{X}(\mathcal{G})$) causes the reverse of the orientations in this net.
So, for the open canonical regions of $\mathcal{X}(\mathcal{G})$ and $\mathcal{X}(\mathcal{G}^{*})$, we have $\overline{\overline{\mathcal{R}} }_{ij}=\overline{\overline{\mathcal{R}} }_{ji}^{*}$ (as point sets). However, the role of $v_{i}$ and $v_{j}^{*}$, and of $\sigma_{1}$ and $\sigma_{2}$
(w.r.t. the orientations of the trajectories) is exchanged. see Fig. \ref{Figure28N}, where the equal angles at $v_{i}$ and $v_{j}^{*}$ in $\overline{\overline{\mathcal{R}} }_{ij}$ and $\overline{\overline{\mathcal{R}} }_{ji}^{*}$ are denoted by $\alpha$.\\
Reasoning as in the case of the function $h$ for $\mathcal{X}(\mathcal{G})$, we find a self indexing smooth function, say $g$, for $\mathcal{X}(\mathcal{G}^{*})$ with the following property:\\
\noindent 
``When traveling along  the boundary of $\overline{\overline{\mathcal{R}} }_{ji}^{*}$ 
, the functional values of $g$
vary strictly from 2 (at the unstable node $v_{i}$) via 1 (at a saddle $\sigma_{1}$ or $\sigma_{2}$) to 0 (at the stable node $v_{j}^{*}$).''

Consider an arbitrary $\mathcal{X}(\mathcal{G})$-trajectory, say $\gamma_{\Delta}$, in $\overline{\overline{\mathcal{R}} }_{ij}$, approachng $v_{i}$ under a positive angle $\Delta$ with the $\mathcal{G}$-edge (=$\mathcal{X}(\mathcal{G})$-trajectory) $v_{i}\sigma_{1}$; see Fig. \ref{Figure28N}. The set of all such trajectories is parametrized by the values of $\Delta$ in the interval ($0,\alpha$) and the functional values of $h$ (or $g$) on $\overline{\overline{\mathcal{R}} }_{ij}$. We map $\gamma_{\Delta}$ onto the half ray $r(x)\exp(i \Delta), x \in \gamma_{\Delta},$ where
\begin{align*}
&r(x)=h(x), \text{ if $x$ is on $\gamma_{\Delta}$ between $v_{i}$ and $p$ (=intersection $\gamma_{\Delta} \cap L_{1}$),}\\
&r(x)=\frac{1}{g(x)},  \text{ if $x$ is on $\gamma_{\Delta}$ between $p$ and $v_{j}^{*}$.}
\end{align*}
In this way, the $R$-orthogonal net of trajectories for $\mathcal{X}(\mathcal{G})$ (=${\rm grad}_{R}h$) and ${\rm  grad}_{R}^{\perp}h$ on 
$\overline{\overline{\mathcal{R}} }_{ij} 
$ can be homeomorphically mapped onto the polar net on the open sector , say $s(\overline{\overline{\mathcal{R}} }_{ij} )$,
in the complex plane as in Fig. \ref{Figure29N}-(a). Here
0 corresponds to $v_{i}$, and $\sigma_{1}^{'}$, $\sigma_{2}^{'}$ (both situated on the unit circle) are related to respectively $\sigma_{1}$ and $\sigma_{2}$.
 
Similarly, the trajectory $\gamma_{\Delta^{*}}^{*}$ in $\overline{\overline{\mathcal{R}} }_{ji}^{*}$ can be mapped onto the half ray
$\frac{1}{r(x)}\exp(i \Delta^{*})
, x \in \gamma_{\Delta^{*}}^{*},$ where $\Delta^{*}$ is the angle at $v_{j}^{*}$ between this trajectory and the $\mathcal{G}^{*}$-edge $v_{j}^{*}\sigma_{1}$, see Fig. \ref{Figure28N}, where $\Delta^{*}=\Delta$ (apart from orientation). Hence, the $R$-orthogonal net of trajectories for $\mathcal{X}(\mathcal{G}^{*})$ (=$-{\rm grad}_{R}h$) and $-{\rm  grad}_{R}^{\perp}h$ on 
$\overline{\overline{\mathcal{R}} }_{ji}^{*} 
$ can be homeomorphically mapped onto the polar net on the sector, obtained from $s(\overline{\overline{\mathcal{R}} }_{ij} )$ by reflection in the real axis. We call this sector $S(\overline{\overline{\mathcal{R}} }_{ji}^{*} )$. Here
0 corresponds to $v_{j}^{*}$, and $(\sigma_{1}^{*})'$, $(\sigma_{2}^{*})'$ (both situated on the unit circle) are related to respectively $\sigma_{1}$ and $\sigma_{2}$.
Reversing the orientations of the polar net in the latter section, we obtain a polar net, oriented as the 
$\mathcal{X}(\mathcal{G})$ (=${\rm grad}_{R}h$) and ${\rm  grad}_{R}^{\perp}h$-trajectories on 
$\overline{\overline{\mathcal{R}} }_{ij} 
$. 
Endowed with this polar net we rename $S(\overline{\overline{\mathcal{R}} }_{ji}^{*} )$ as $S(\overline{\overline{\mathcal{R}} }_{ij} )$. 
Apparently, the polar nets on $s(\overline{\overline{\mathcal{R}} }_{ij} )$ and $S(\overline{\overline{\mathcal{R}} }_{ij} )$ correspond under the inversion $z \to \frac{1}{z}$. Compare Fig. \ref{Figure29N}-(a),(b). In the same way, we map a neighbouring region $\overline{\overline{\mathcal{R}} }_{ij'}$  as in Fig. \ref{Figure28N}, homeomorphically onto the sector $s(\overline{\overline{\mathcal{R}} }_{ij'})$ in Fig. \ref{Figure29N}-(a) and also onto $S(\overline{\overline{\mathcal{R}} }_{ij'})$ in Fig. \ref{Figure29N}-(c). Repeating this procedure we are able to map all canonical regions of $\mathcal{X}(\mathcal{G})$ onto (the closures of) sectors of the types $s(\cdot)$ and $S(\cdot)$ in such a way that together they cover -for each value of $i$ and $j$ a copy of the complex plane. (Compare also Fig. \ref{Figure18N} and \ref{Figure19N}).

In analogy with Remark \ref{R6.20}, we consider the reduced torus $\check{T}= T\backslash \{\mathcal{G} \wedge \mathcal{G}^{*} \text{-vertices}\}$, and on $\check{T}$ the covering by open neighborhoods 
$$
\{ F^{*}_{v_{i}} \backslash v_{i}, F_{v^{*}_{j}} \backslash v^{*}_{j} \}, i,j=1. \cdots ,r,
$$
, where $F^{*}_{v_{i}}$ and $F_{v^{*}_{j}}$ stand for the basins of $\mathcal{X}(\mathcal{G})$ for respectively $v_{i}$ and $v^{*}_{j}$. Again, only intersections of the type $(F^{*}_{v_{i}} \backslash v_{i})\cap (F_{v^{*}_{j}} \backslash v^{*}_{j})$ are possibly non-empty. Even so, such an intersection consists of the disjoint union of regions of the type $\overline{\overline{\mathcal{R}} }_{ij}$, say $\overline{\overline{\mathcal{R}} }_{ij}^{1}, \cdots , \overline{\overline{\mathcal{R}} }_{ij}^{s}$, where $s$ is the amount of vertices $v_{i}$(vertices $v^{*}_{j}$) in the $\Pi$-walks of $F_{v^{*}_{j}}$  (of $F^{*}_{v_{i}}$). Note that at $v_{i}$, (resp. $v^{*}_{j}$) these regions $\overline{\overline{\mathcal{R}} }_{ij}^{k}$, are endowed with the anti-clockwise (clockwise) cyclic orientation, and are separated by regions not of this type; compare Remark \ref{R6.8N} and Subsection 3.2. 

Now, we proceed as in Remark \ref{R6.20}: The open covering of $\check{T}$ provides this manifold with a complex analytic structure, exhibiting coordinate transfomations  
$$
s(\overline{\overline{\mathcal{R}} }_{ij}^{k}) \leftrightarrow S(\overline{\overline{\mathcal{R}} }_{ij}^{k}), i,j=1. \cdots ,r,
$$
induced by the inversion $z \to \frac{1}{z}$. We pull back the restrictions of the function $z$ (resp. $\frac{1}{z}$) on the various sectors $s(\overline{\overline{\mathcal{R}} }_{ij}^{k})$, resp. $S(\overline{\overline{\mathcal{R}} }_{ij}^{k})$  to $\check{T}$. By glueing all canonical regions for $\mathcal{X}(\mathcal{G})$ along the trajectories in their common boundaries, we construct a complex analytic function on $\check{T}$. Continuous extension to $T$, yields a meromorphic function, say $f$ on $T$, with $r$ simple zeros (poles) at $v_{i}$ ($v_{j}^{*}$) and $2r$ simple saddles at $\sigma_{1}, \cdots ,\sigma_{2r}$. Since $\overline{\mathcal{N}}(z)=-z; \overline{\mathcal{N}}(\frac{1}{z})=-\frac{1}{z}$, we find $\mathcal{X}(\mathcal{G})=\overline{\overline{\mathcal{N}} }(f)$, thus $\mathcal{G}=\mathcal{G}(f)$. This proves Lemma \ref{NT8.1}.

  \begin{figure}[h!]
\begin{center}
\includegraphics[scale=0.6]{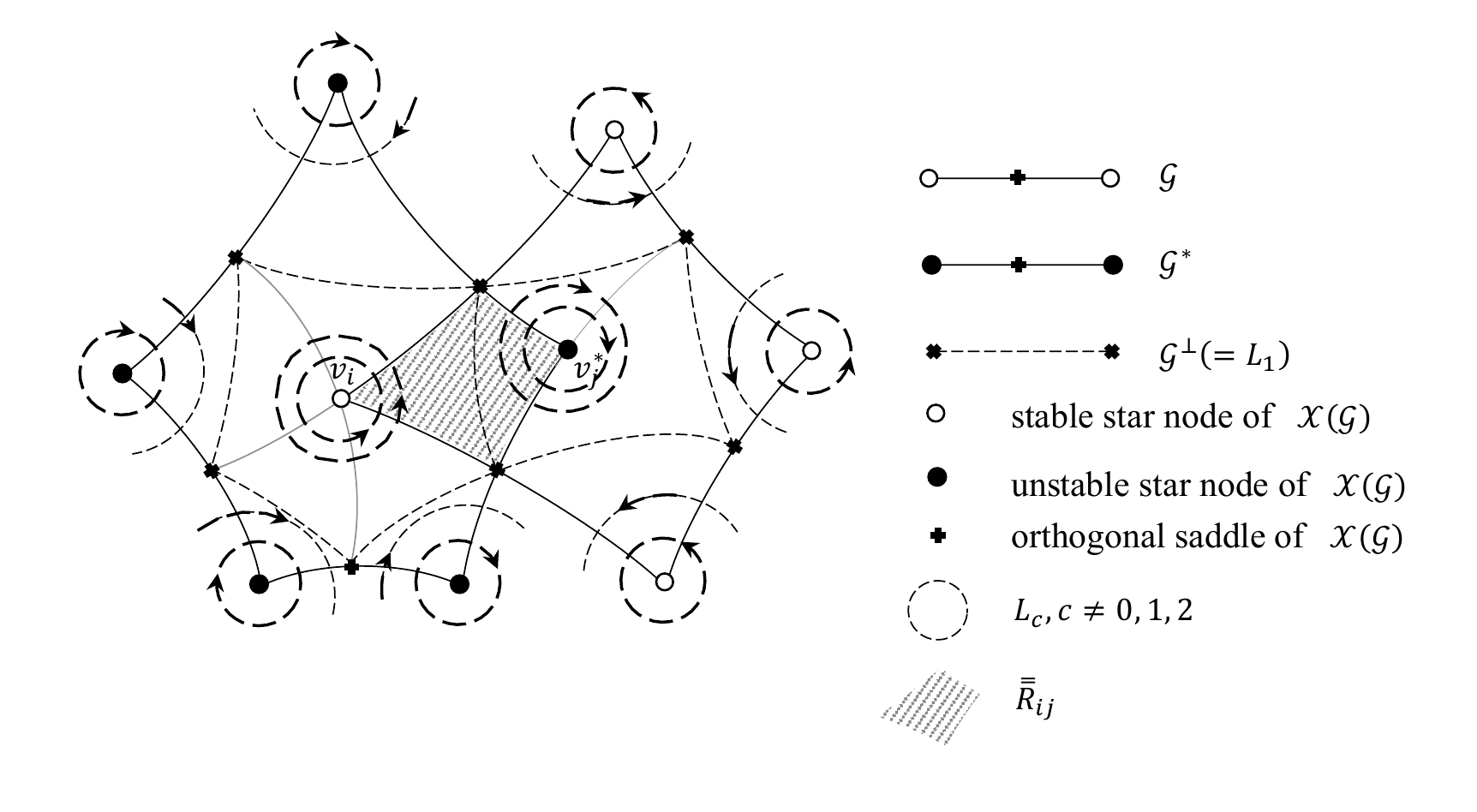}
\caption{\label{Figure8.4} The graphs $\mathcal{G} \wedge \mathcal{G}^{*}$, $\mathcal{G}^{\perp}$ and some level sets for $h$.}
\end{center}
\end{figure}

We combine this result together with results obtained in the preceding sections as follows:

\begin{theorem}
\label{NT8.2} $($Representation of structurally stable elliptic Newton flows by graphs.$)$\\
Up to conjugacy $(\sim)$ between flows and equivalency $(\sim)$ between graphs, the structurally stable Newton flows of $r$-th order 
are 1-1 represented by the Newton graphs of order $r$.
\end{theorem}
\begin{proof}
Follows from  Theorem \ref{T6.10N}, Corollary \ref{NC7.15} and Theorem \ref{NT8.1}.
\end{proof}

\begin{figure}[h!]
\begin{center}
\includegraphics[scale=0.58]{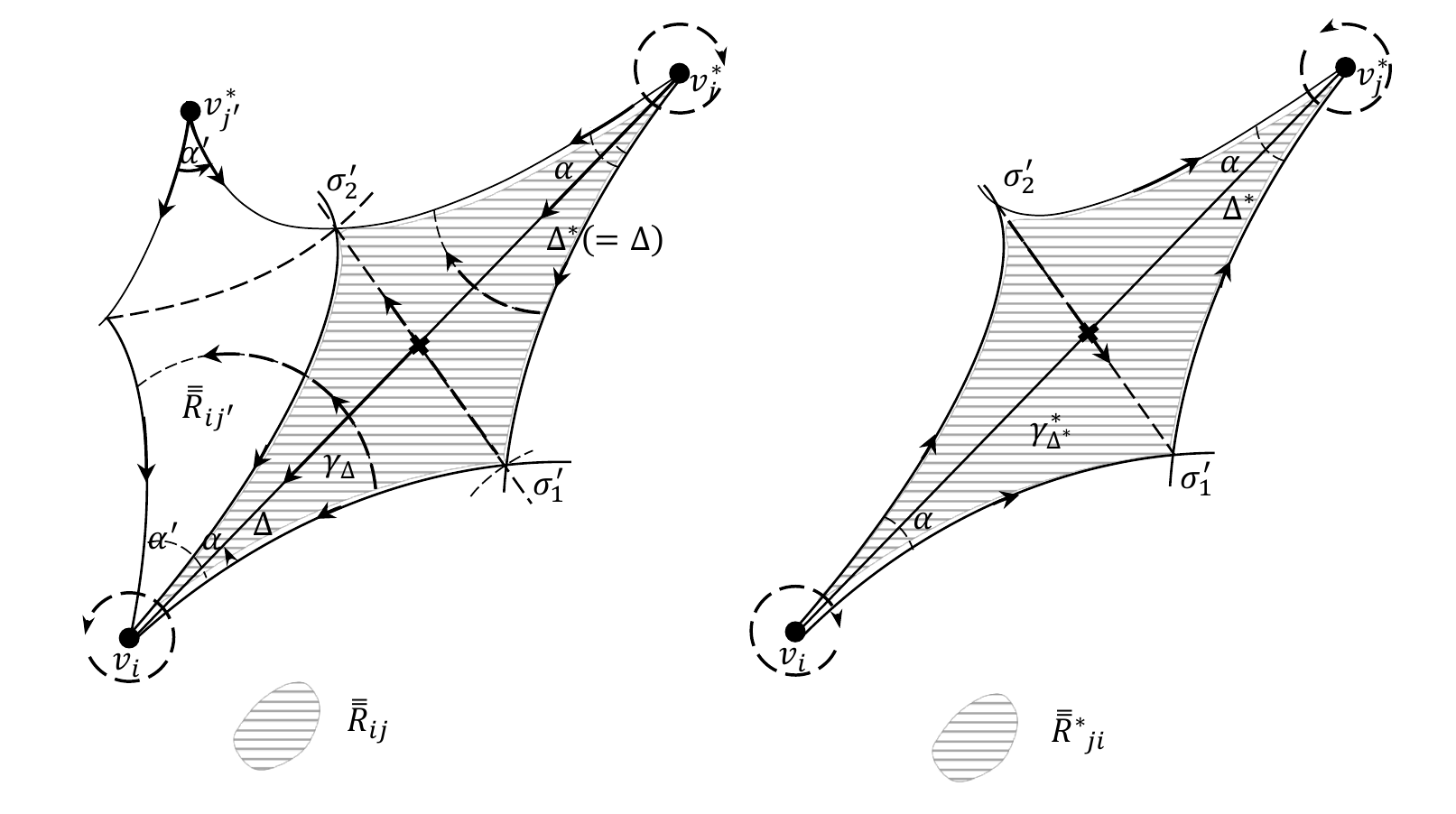}
\caption{\label{Figure28N}$\overline{\overline{\mathcal{R}} }_{ij}$ and $\overline{\overline{\mathcal{R}} }_{ij}^{*}$}
\end{center}
\end{figure}

\newpage

\begin{figure}[h!]
\begin{center}
\includegraphics[scale=0.50]{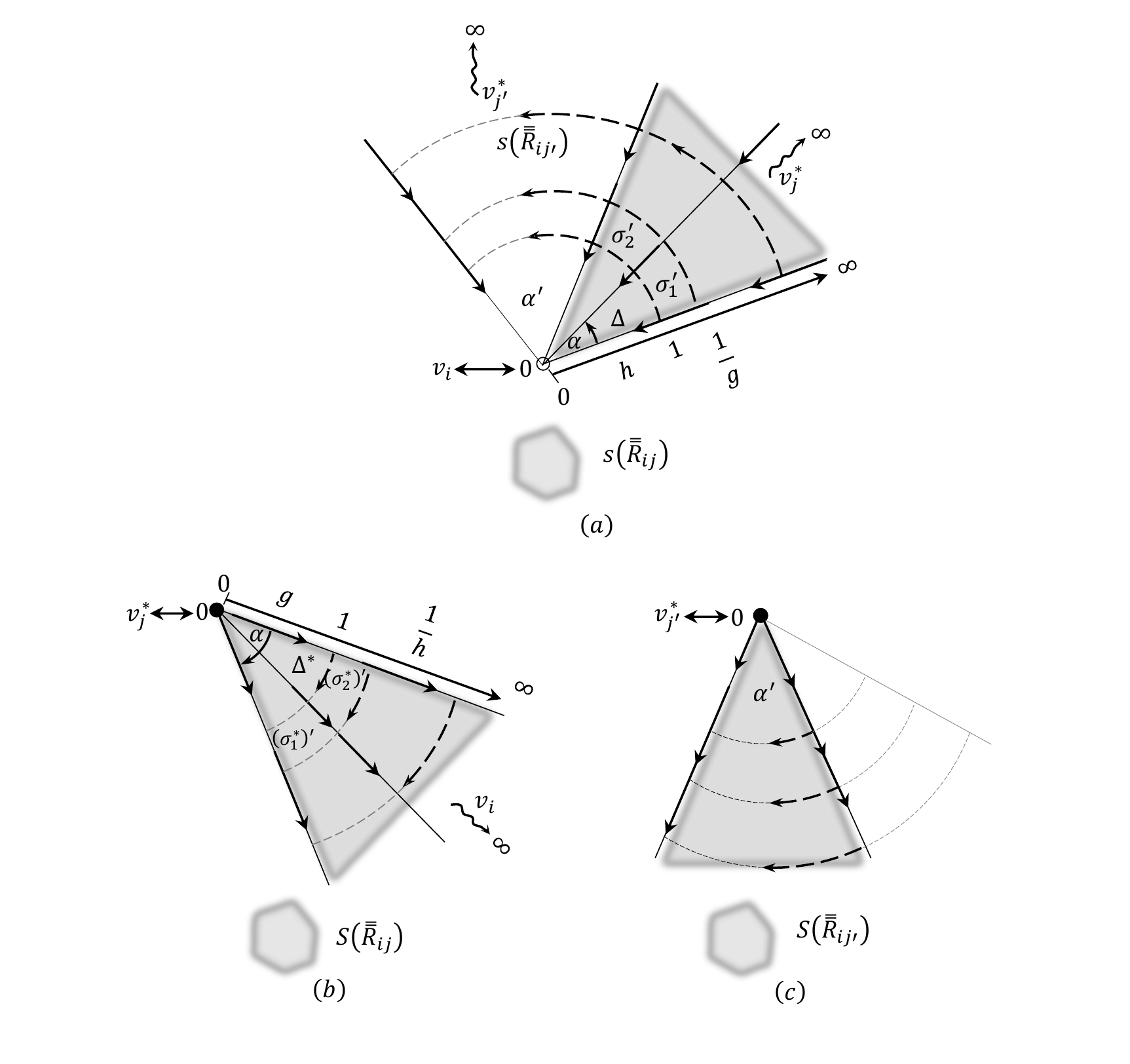}
\caption{\label{Figure29N}$s(\overline{\overline{\mathcal{R}} }_{ij})$, $S(\overline{\overline{\mathcal{R}} }_{ij})$ and $s(\overline{\overline{\mathcal{R}} }_{ij'})$, $S(\overline{\overline{\mathcal{R}} }_{ij'})$}
\end{center}
\end{figure}

\newpage

\section{Final remarks}
\label{sec13} 

\subsection{Rational versus elliptic Newton flows.}
\label{ss13.1}

Our study is inspired by the analogy between rational and elliptic functions. We raised the question whether, and -even so - to what extent, this analogy persists in terms of the corresponding Newton flows (on resp. the Riemann sphere $S^{2}$  and the torus $T$). An affirmative answer to this question is given by comparing the {\it characterization, genericity, classification and representation aspects} of rational Newton flows (see Theorem 2.1 in \cite{HT1}) with their counterparts as described in Theorem \ref{T1.2}, 
Theorem \ref{T6.10N} and Theorem  \ref{NT8.2}.

More in particular, this analogy becomes manifest when we look at  the special case of {\it balanced rational} Newton flows of order $r \geqslant 1$. By these, we mean structurally stable flows of the form  $\overline{\overline{\mathcal{N}}}(\frac{p_{n}}{q_{m}})$, with $p_{n} , q_{m}$ two co-prime polynomials of degrees respectively $n, m 
, |n-m| \leqslant 1,  r =\max \{n, m\}$. Such flows admit $2r$ star nodes ($r$ stable and $r$ unstable) together with $2r-2$ orthogonal saddles. [Note that at $z= \infty$ (north pole) there is an unstable node if $n=m+1$, a stable node if $m=n+1$, and a saddle if $m=n$]. Due to the duality property\footnote{\label{FTNT17}Duality for rational Newton flows is easily verified, see (\ref{vgl2x}).} 
, the transition 
$\frac{p_{n}}{q_{m}} \leftrightarrow \frac{q_{m}}{p_{n}}$
causes the reverse of orientations of the trajectories of $\overline{\overline{\mathcal{N}}}(\frac{p_{n}}{q_{m}})$ and $\overline{\overline{\mathcal{N}}}(\frac{q_{m}}{p_{n}})$. So, these flows maybe be considered as equal and we assume
$n \geqslant m$. Now, the oriented sphere graph $G(\frac{p_{n}}{q_{m}})$ for $\overline{\overline{\mathcal{N}}}(\frac{p_{n}}{q_{m}})$ can be defined (in strict analogy with Definition  \ref{D6.1}) as a connected, cellularly embedded multigraph with $r$ vertices, $2r-2$ edges and $r$ faces; apparently, also:   $G(\frac{q_{m}}{p_{n}}) =-G^{*}(\frac{p_{n}}{q_{m}})$
holds. As in the elliptic case, it can be proved  that $G(\frac{p_{n}}{q_{m}})$ fulfils both the {\it E-} and the {\it A-property}. (However, in this special case it is found that the later property already implies the first one). Subsequently, it is shown that any cellularly embedded  multigraph in $S^{2}$  with $r$ vertices, $2r-2$ edges and $r$ faces, admits the {\it A-property} iff certain (Hall) inequalities are satisfied. Altogether, this leads to a concept of Newton graph that is formally the same as the concept of Newton graph in Definition  \ref{ND7.12}. In particular, classification and representation results, similar to Theorem \ref{T6.10N} and Theorem  \ref{NT8.2}, are derived (cf. \cite{JJT3}, \cite{JJT4}).

We conclude that there is a striking analogy between the balanced Newton flows and elliptic Newton flows, both of order 
$r$.
(Note that an elliptic Newton flow of order 1 is not defined, whereas a balanced Newton flow of order 1 is just the North-South flow (cf. Fig. 7
and  8 in \cite{HT1} for $n=1$).

Finally, we note that - as in the elliptic case - for lower values of $r$ a list of all possible (up to conjugacy and duality) balanced Newton flows, represented by their graphs, is available. For example, see Fig.\ref{Figure57}, where the pictures of the graphs $G_{r}(\frac{p_{n}}{q_{m}})$ and $G_{r}^{*}(\frac{p_{n}}{q_{m}})$, $r=2,\;3$, suggest that the conditions A1, A2, A3, in Definition \ref{ND7.6} are indeed fulfilled. The proof that these graphs are the only possibilities, based on the Representation Theorem for rational Newton flows (compare \cite{JJT3}), is omitted.

\begin{figure}[htbp]
\centering
\includegraphics[width=5.5in]{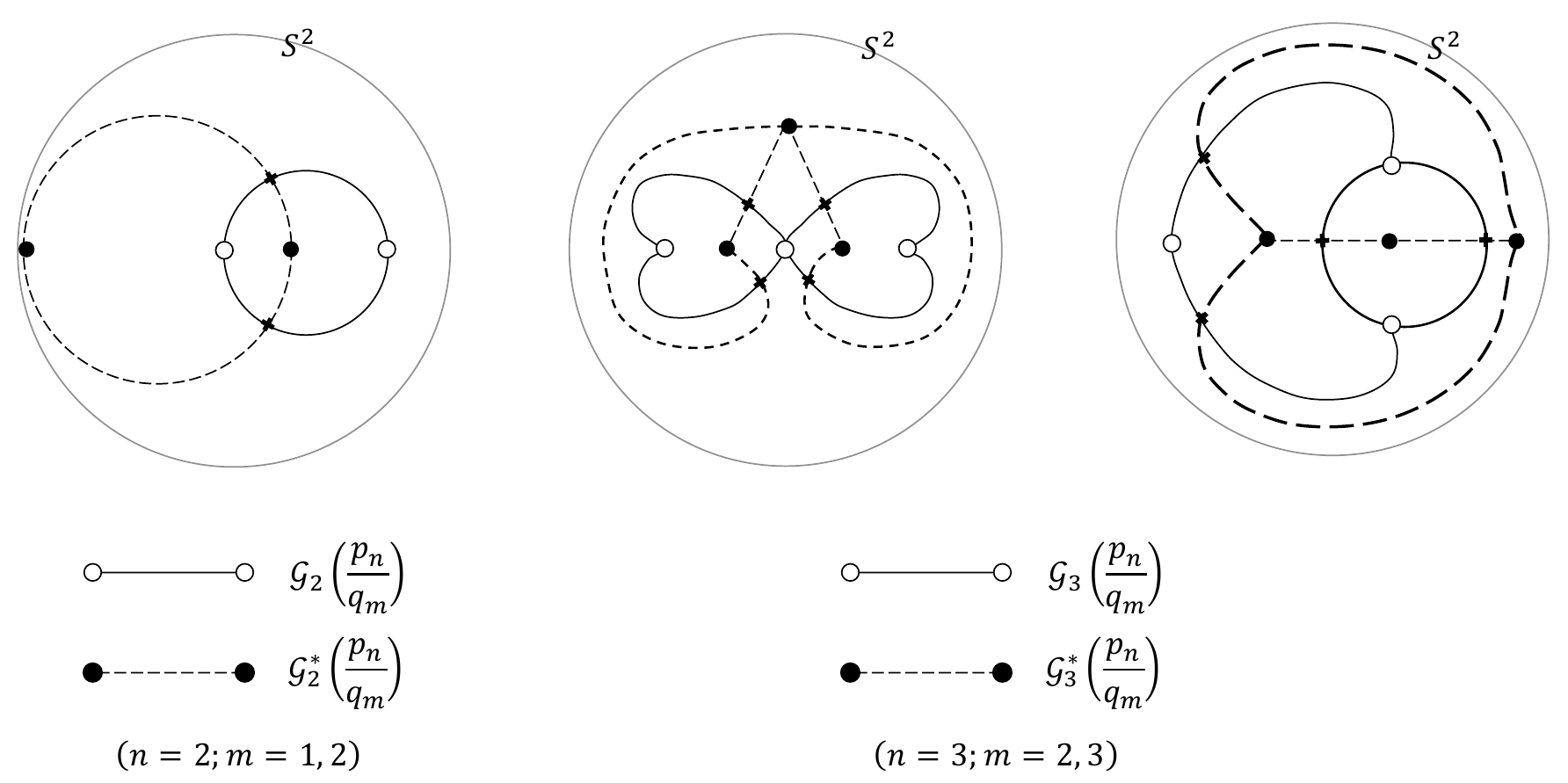}
\caption{The different graphs for balanced rational Newton flows of order $r=2, 3$.}
\label{Figure57}
\end{figure} 

\subsection{Complexity aspects}
\label{ss13.2}

We indicate the existence of a ``good'' (i.e., polynomial) algorithm deciding whether a given cellularly embedded torodial graph $\mathcal{G}_{r}$ is a Newton graph or not. To this aim, we check both the {\it E-} and {\it A-property}.\\

\noindent
\underline{{\it E-property}}: Use that the graphs (facial walks) $\partial F_{j}$ are Eulerian iff all vertices have even degree.\\
\noindent
\underline{{\it A-property}}: Let $B$ be a finite bipartite graph with bipartition $(X, Y)$, and denote for any subset $S$ in $X$ the neighbour set in $Y$ by $N(S)$. We consider the so called Strong Hall Property (cf. \cite{Coleman-E-G86}):
\begin{equation}
\label{vgl37}
|S| < |N(S)| \text{, for all nonempty } S \subset X.
\end{equation}
For each {\it bipartite} graph, obtained from $(B, X, Y)$ by adding one vertex ($p$) to $X$ and one edge
which joins $p$ to an $Y$-vertex, we also consider the Hall property (cf. \cite{Bondy-M76}:
\begin{equation}
\label{vgl38}
|\check{S}|  \leqslant |N(\check{S})| \text{, for all subsets } \check{S} \text{ of }X \cup \{p\}.
\end{equation}
It is easily shown that (\ref{vgl37}) and (\ref{vgl38}) are equivalent, and thus:  
Because the verification of (\ref{vgl38}) is possible in polynomial time (cf. \cite{Bondy-M76}), this is also true for (\ref{vgl37}).  
Now, we select an arbitrary $\mathcal{G}_{r}$-face $F_{j}$, say $F_{r}$ , and specify $(B, X, Y)$ by $X=\{F_{1}, \cdots , F_{r-1}\}$. $Y=V(\mathcal{G}_{r})$, where adjacency is defined by inclusion. The inequalities in the right hand side of Lemma \ref{NL7.10}
take the form (\ref{vgl37}) for all non-empty $J$ in $\{1, \cdots, r-1 \}$, and considering all possible choices for $F_{j}$, we are done.


\begin{thebibliography}{10}



\bibitem{A/S} Abramowitz, A., Stegun, I.A. (eds): 
{\em Handbook of Mathematical Functions.} 
\newblock {\em Dover Publ. Inc. (1965). } . 

\bibitem{ALGM} Andranov, A.A., Leontovich, E.A., Gordon, I.I., and Maier, A.G.:
{\em Theory of bifurcations of dynamical systems on a plane.}
\newblock {\em John Wiley, New York 1973.} 

	 
\bibitem{THBS} Twilt, F.; Helminck, G.F.; Snuverink, M.; Van den Burg, L.: {\em Newton flows for elliptic functions: a pilot study}, Optimization 57 (2008), no. 1, 113--134. 

 \bibitem{HT1} Helminck, G.F., Twilt, F.:
{\em Newton flows for elliptic functions I, Structural stability: characterization \& genericity}, submitted.
	 

\bibitem{HT3} Helminck, G.F., Twilt, F.:
{\em Newton flows for elliptic functions III, Classification of $3^{rd}$ order Newton graphs}, submitted.

 \bibitem{Bondy-M76}  Bondy, J.A. and  Murty, U.S.R.: {\em Graph Theory with Applications}, Macmillan \& Co, London, 1976[2].
 
 \bibitem{Coleman-E-G86}  Coleman, T.F., Edenbrandt, A. and Gilbert, J.R.: {\em Predicting Fill for Sparse Orthogonal Factorization}.  \newblock {\em Jour. Assoc. Comp. Machinery}, Vol.33, No 3, July 1986. pp. 517-532.

 \bibitem{Gib} Giblin, P.J.: {\em Graphs, surfaces and homology}, John Wiley and Sons (1977).
 
 
\bibitem{Har} Harary, F.: {\em Graph theory}, Addison-Wesley Publ. Co. Inc. (1969).
 
 \bibitem{Hart} Hartman, P. {\em Ordinary Differential Equations},  Birkh\"auser (1982).
 
\bibitem{Hir} Hirsch, M.W.: 
{\em Differential Topology.} 
\newblock {\em Springer Verlag (1976) } . 

\bibitem{JJT1} Jongen, H.Th., Jonker, P., Twilt, F.:
{\em Non linear Optimization in $\R^n$: Morse Theory, Chebyshev Approximation, Transversality, Flows, Parametric Aspects.} 
\newblock {\em Kluwer Ac. Publ., Dordrecht, Boston (2000)} . 

\bibitem{JJT2} Jongen, H.Th., Jonker, P., Twilt, F.:
{\em The continuous Newton method for meromorphic functions.} 
\newblock {\em In: Geometric Approaches to Differential Equations (R. Martini, ed), Lect. Notes in Math., Vol. 810, , pp. 181-239. Springer Verlag(1980)} . 

\bibitem{JJT3} Jongen, H.Th., Jonker, P., Twilt, F.: 
{\em The Continuous, desingularized Newton method for meromorphic functions.} 
\newblock {\em Acta Applicandae Mathematicae }13, Nos. 
	  1 and 2, pp. 81-121 (1988) . 

\bibitem{JJT4} Jongen, H.Th., Jonker, P., Twilt, F.: 
{\em On the classification of plane graphs 
	  representing structurally stable Rational Newton flows. } 
\newblock {\em Journal of 
	  Combinatorial Theory,} Series B, Vol. 51, No.2, pp. 256-270 (1991) . 

	  
\bibitem{Mang}	Mangasarian, O.I.: {\em Nonlinear programming}, McGraw-Hill Book Co. (1969).

\bibitem{M1} Markushevich, A.I.: 
{\em Theory of Functions of a Complex Variable, Vol. II, } 
\newblock {\em   Prentice Hall (1965)} . 

\bibitem{M2} Markushevich, A.I.: 
{\em Theory of Functions of a Complex Variable, Vol. III, } 
\newblock {\em   Prentice Hall (1967)} . 


\bibitem{Mirsky} Mirsky, L. {\em Transversal theory: an account of some aspects of combinatorial mathematics}
\newblock {\em  New York: Academic Press (1971)} .

\bibitem{MoTh} Mohar, B.;Thomassen, C. : {\em Graphs on surfaces}. John Hopkins Studies in the Mathematical Sciences. John Hopkins University Press, Baltimore, MD, 2001.


\bibitem{Peix1} Peixoto, M.M.: 
{\em  Structural stability on two-dimensional manifolds.} 
\newblock {\em Topology, } 1, pp. 101-120 (1962). 

\bibitem{Peix2} Peixoto, M.M.: 
{\em  On the classification of flows on 2-manifolds.} 
\newblock {\em In:Dynamical Systems, M.M. Peixoto, ed., pp. 389-419, Acad. Press, NewYork (1973)} . 


\bibitem{Smale}  Smale, S.: {\em On Gradient Dynamical Systems.}
\newblock {\em Ann. of Math.}, Second Series, Vol. 74, No. 1 (1961), pp. 199-206.


\end{thebibliography}
\end{document}